\documentclass[]{amsart}
\usepackage{amsmath}
\usepackage[english]{babel}
\usepackage[utf8]{inputenc}
\usepackage{amsfonts}
\usepackage{amsthm}
\usepackage{amssymb}
\usepackage{graphicx}
\usepackage{color}
\usepackage[colorlinks=true,linkcolor=blue]{hyperref}

\usepackage[margin=1.2in,marginparwidth=1.5cm, marginparsep=0.5cm]{geometry}

\usepackage{enumitem}
\setlist[itemize]{leftmargin=5.5mm}

\makeatletter
\renewcommand*{\eqref}[1]{%
  \hyperref[{#1}]{\textup{\tagform@{\ref*{#1}}}}%
}
\makeatother
\numberwithin{equation}{section}

\newtheorem{thm}{Theorem}[section]
\newtheorem{defn}[thm]{Definition}

\newtheorem{prop}[thm]{Proposition}
\newtheorem{cor}[thm]{Corollary}
\newtheorem{lem}[thm]{Lemma}

\newtheorem{rmk}[thm]{Remark}
\theoremstyle{definition}

\DeclareMathOperator{\fX}{\mathfrak{X}}

\DeclareMathOperator{\E}{\mathbb{E}}
\DeclareMathOperator{\bS}{\mathbb{S}}

\DeclareMathOperator{\N}{\mathbb{N}}

\DeclareMathOperator{\R}{\mathbb{R}}
\DeclareMathOperator{\cF}{\mathcal{F}}
\DeclareMathOperator{\cG}{\mathcal{G}}

\DeclareMathOperator{\cP}{\mathcal{P}}

\DeclareMathOperator{\cA}{\mathcal{A}}
\DeclareMathOperator{\cL}{\mathcal{L}}

\DeclareMathOperator{\cB}{\mathcal{B}}

\DeclareMathOperator{\fe}{\mathfrak{e}}
\DeclareMathOperator{\fE}{\mathfrak{E}}

\DeclareMathOperator{\bP}{\mathbb{P}}

\DeclareMathOperator{\C}{\mathbb{C}}

\newcommand{\pd}[2]{\frac{\partial #1}{\partial #2}}
\newcommand{\pdsup}[3]{\frac{\partial^{#3} #1}{\partial #2^{#3}}}

\newcommand{\der}[2]{\frac{d #1}{d #2}}
\newcommand{\dersup}[3]{\frac{d^{#3} #1}{d #2^{#3}}}

\newcommand{\Norm}[2]{\left\Vert #1 \right\Vert_{#2}}

\makeatletter
\def\l@subsection{\@tocline{2}{0pt}{2.5pc}{}{}}
\def\l@subsubsection{\@tocline{2}{0pt}{5pc}{}{}}
\makeatother

\title[Time Changed Spherical Brownian Motions with Drifts]{Time Changed Spherical Brownian Motions\\ with Longitudinal Drifts}
\author[G. Ascione]{Giacomo Ascione$^\ast$}
\address{$^\ast$ Scuola Superiore Meridionale}
\email{g.ascione@ssmeridionale.it}
\author[A. Vidotto]{Anna Vidotto$^\diamond$}
\address{$^\diamond$ Dipartimento di Matematica e Applicazioni, Università di Napoli Federico II}
\date{\today}
\email{anna.vidotto@unina.it}
\begin{document}
	\maketitle
	\begin{abstract}
		In this paper, we consider the time change of the diffusion process on the 2-dimensional unit sphere generated by the Laplace-Beltrami operator, perturbed by means of a longitudinal vector field. First, this is done by addressing the problem of finding strong solutions to suitable time-nonlocal Kolmogorov equations, via a spectral decomposition approach. Next, the desired process is constructed as the composition of the aforementioned diffusion process and the inverse of a subordinator, and it is used to provide a stochastic representation of the solution of the involved time-nonlocal Kolmogorov equation, which in turn leads to the spectral decomposition of its probability density function. 
A family of operators induced by the process is then adopted to provide very weak solutions of the same time-nonlocal Kolmogorov equation with much less regular initial data. From the spectral decomposition results we also get some bounds on the speed of convergence to the stationary state, proving that the process can be considered an anomalous diffusion. These results improve some known ones in terms of both the presence of a perturbation and the lack of regularity of the initial data.  \\
{\bf Keywords:} Spherical harmonic; Anomalous diffusion; Time-nonlocal Kolmogorov equation; Subordinator.\\
{\bf MSC 2020 classification:} 35R11, 60K15, 60J60, 33C55.
	\end{abstract}

	\section{Introduction}
In recent years, anomalous diffusions have been observed in several physical and biological phenomena. The first instance of anomalous diffusion was found in \cite{richardson1926atmospheric} in the context of turbulent flows. Sub-diffusions have been further observed, for instance, in the motion of the mRNA inside live \textit{E. Coli} cells \cite{golding2006physical} and, more in general, in the motion of tracer particles in the cytoplasm of living mammals cells \cite{sabri2020elucidating}. Power-like relaxation has been also observed in micellar solutions \cite{jeon2013anomalous}, while superdiffusive behaviours have been noticed in the motion of tracer particles transported by molecular motors \cite{bruno2009transition}. For a survey on anomalous diffusions, we refer to \cite{metzler2014anomalous}.

Power-like mean square displacements and relaxation dynamics have been modelled by means of fractional differential equations. Fractional calculus has been developed since the end of the Seventeenth Century, but the use of fractional derivatives \textit{exploded} the last century, starting from \cite{caputo1967linear}, see, just for example \cite{mainardi1996fractional,barkai2001fractional} and references therein. The first instance of the connection between stochastic processes and time-fractional equations has been observed in \cite{saichev1997fractional}, where the solution is obtained by means of a Bochner-like subordination, i.e. by composing the Feller process generated by the diffusion operator (which was a L\'evy stable motion) with the inverse of an independent stable subordinator. This result has been generalized in \cite{baeumer2001stochastic}, by considering time-fractional abstract Cauchy problems and applying the time-change directly to the semigroup generated by the considered operator. Furthermore, thanks to the recognition of the Mittag-Leffler function as both the Laplace transform of the inverse stable subordinator \cite{bingham1971limit} and the eigenfunction of the Caputo fractional derivative, in \cite{leonenko2013fractional} the authors provided a spectral decomposition formula to recognize the strong solutions of a family of time-fractional partial differential equations related to processes of the Pearson diffusion family. Let us stress that processes that are time-changed by means of inverse stable subordinators usually exhibit the desired power-law relaxation property, see \cite{ascione2023uniform} and references therein. For further details on stochastic models for fractional differential equations we refer to \cite{meerschaert2019stochastic}.

However, phenomena like non-power-law relaxation property and mean square displacement asymptotics have been also observed. This is the case, for instance, of colloidal liquid-glass transition \cite{sperl2005nearly}, that exhibits a ultraslow behaviour, i.e. the mean square displacement is logarithmic in time, and whose stochastic model has been provided e.g.~in \cite{meerschaert2006stochastic} (see also \cite{liang2019survey} for other examples). Unfortunately, the power-like asymptotic of the Mittag-Leffler function is not enough to describe such logarithmic relaxation phenomena, hence it is necessary to consider a more general setting than the one provided by fractional calculus. A \textit{general fractional calculus} was introduced in \cite{kochubei2011general}, while the connection between these newly introduced Caputo-like derivatives and time-changed semigroups (and Feller processes) has been investigated in \cite{toaldo2015convolution,chen2017time}. Despite we do not have an explicit expression for the Laplace transform  and the moment generating function of a general inverse subordiator, these quantities are still shown to be eigenfunctions of the considered time-nonlocal operators, see \cite{meerschaert2019relaxation,ascione2021abstract}. 
 Thanks to this observation, the spectral decomposition methods first introduced in \cite{leonenko2013fractional} can be extended to the general fractional calculus, as done for instance in \cite{ascione2021time,ascione2022non2}.

The spectral decomposition strategy presented in \cite{leonenko2013fractional,ascione2021time,ascione2022non2} for Pearson diffusions and some special birth-death processes relied on the knowledge of the eigenfunctions and the eigenvalues of their generators. Another interesting case in which the spectrum is fully known is the Laplace-Beltrami operator on a compact manifold without boundary, and a prototype of such a setting is the two-dimensional unit sphere $\bS^2$. Diffusions on the sphere are used, for example, to smooth 3D surfaces \cite{Bu:04} (in particular, via the parabolic smoothing property) or to represent migration via particle motion models \cite{brillinger1998elephant}.
Fractional diffusive behaviours on $\bS^2$ have been studied in \cite{d2014time,DLO:16} (and extended to general compact manifolds in \cite{d2016fractional}), while in the general fractional calculus setting a first approach has been considered in \cite{d2022models}. In these works, the underlying stochastic processes are constructed through the time change and the subordination of a standard spherical Brownian motion, as introduced in \cite{yosida1949brownian}.

In this paper, we consider a further generalization of this spherical model; precisely, we consider the case of a spherical Brownian motion with longitudinal drift, see \cite[\S 7]{LGY:86}. This is motivated by the possibility of modeling the diffusion on the sphere of particles that are subjected to longitudinal currents or winds which force their direction, as for instance describing \textit{preferential routes} in migration models. As a \textit{side effect}, we improve the results in \cite{d2022models,d2014time}, by providing strong solutions of the involved time-nonlocal equations for less regular initial data.

In particular, we will consider the following time-nonlocal Cauchy problem, also called time-nonlocal Kolmogorov equation:
\begin{equation}\label{eq:bKintro}
	\begin{cases}
		\displaystyle \dersup{u}{t}{\Phi}(t)=(\Delta_{\bS^2}+X_\mu) u(t) & t>0, \\
		u(0)=f,
	\end{cases}
\end{equation}
where $\dersup{}{t}{\Phi}$ is a general fractional derivative operator (see Definition \ref{defn:strongder}) defined starting from the Laplace exponent of a subordinator (see Definition \ref{defn:subordinator}), $\Delta_{\bS^2}$ is the Laplacian on the unit two-dimensional sphere $\bS^2$, see \eqref{eq:Laplacian}, and $X_\mu$ is a longitudinal drift with intensity $\mu$ (see Definition \ref{def:drift}), that is $X_\mu=\mu \frac{\partial}{\partial \varphi}$ once we use the standard spherical coordinate system where $\varphi \in [0,2\pi)$ is the longitude. For $\mu=0$ we get a specific case of the equation considered in \cite{d2022models}, which reduces, if we consider stable subordinators, to the fractional diffusion equation on the sphere considered in \cite{d2014time}. 


The main results of the paper are as follows. In Theorem \ref{thm:specdecKolm} we provide an explicit strong solution to \eqref{eq:bKintro} via spectral decompositions results for any initial data $f \in H^s(\bS^2)$, $s>1$, that is  
	$f \in L^2(\bS^2)$ and
	\begin{equation*}
		\sum_{\ell=0}^{+\infty}\sum_{m=-\ell}^{\ell}|f_{\ell,m}|^2\ell^{2s}<+\infty,
	\end{equation*}
	thus improving the results in \cite{d2014time,d2022models} in the case of $\mu=0$, in which the proofs required $s>3$. Let us stress that we cannot do better than this, as we show with the example in Proposition \ref{prop:counter}.
Furthermore, in Section \ref{sec:sphericalB}, we study in detail the time-changed spherical Brownian motion with longitudinal drift, that we denote by $W_\Phi^\mu$. This allows us to provide a stochastic representation result for strong solutions of \eqref{eq:bKintro}, together with an explicit form of the fundamental solution of such equation. Limit and stationary distributions are also provided.
	In particular, in Proposition \ref{prop:stat} we prove the ergodicity in total variation of the process $W_\Phi^\mu$ while in Proposition \ref{prop:speed} we provide a result on the speed of convergence to the stationary state in a Fortet-Mourier-type metric: in the case $\mu=0$, we have a uniform non-exponential upper and lower bound, that implies the anomalous diffusive behavior of $W_\Phi^0$. For $\mu\ne0$, we still have an anomalous diffusive behavior but we are not able to provide a uniform lower bound if the initial position is exactly on the equator.
Finally, in Theorem \ref{thm:specdecgen}, thanks to the extension of the time-changed semigroup to $L^2(\bS^2)$, we are able to provide explicit very weak solutions, in the sense of \cite{Chen2018}, to \eqref{eq:bKintro} for any initial data $f \in L^2(\bS^2)$, thus definitely relaxing the smoothness assumptions.

\subsection*{Plan of the paper}
Section \ref{sec:PR} is devoted to the preliminary results that are needed in order to understand our model and the objective of the present paper: we first define the Brownian motion with longitudinal drift on the sphere, Section \ref{subs:BMdrift}, and then we recall basic notions related to subordinators, inverse subordinators and Caputo-type nonlocal operators, Section \ref{subs:subecc}. 
In Section \ref{sec:MR} we present our main results. In particular, in Section \ref{sec:strongsol} we obtain the strong solutions of the time-nonlocal Kolmogorov equation together, while in Section \ref{sec:sphericalB} we proceed with an accurate study of the time-changed spherical Brownian motion with longitudinal drift, which is the stochastic representation of such strong solutions. In Section \ref{subs:weak} we are able to provide some very weak solutions of the time-nonlocal Kolmogorov equation under extremely mild assumptions on the initial data. Most of the proofs can be found immediately after the statements, however the more technical or standard ones have been placed in the appendix.

\subsection*{Acknowledgements}
The authors are members of INdAM-GNAMPA. The authors wish to thank Enrica Pirozzi for useful discussions. 
G.A.~has been supported by PRIN 2022 project 2022XZSAFN. A.V.~has been supported by the co-financing of the European Union - FSE-REACT-EU, PON Research and Innovation 2014--2020, DM 1062/2021.

\section{Preliminary results}\label{sec:PR}
	\subsection{The Brownian motion with longitudinal drift on the sphere}\label{subs:BMdrift}
	\subsubsection{Longitudinal drifts and spherical harmonics}
	Let us denote by $\bS^2$ the $2$-dimensional unit sphere, equipped with the Riemannian metric $g$ obtained by push-forward of the Euclidean metric via spherical coordinates. Since we will work with complex valued functions and their gradients, let us extend the Riemannian metric $g$ as a symmetric bilinear form on the complexification $T^{\C}\bS^2$ of the tangent bundle on $\bS^2$. We denote the smooth manifold equipped with the structure that determines $g$ as $(\bS^2,g)$. We say that $\cA$ is an atlas on $(\bS^2,g)$ if it is an atlas on $\bS^2$ determining the same smooth structure as in $(\bS^2,g)$. Let $\fX(\bS^2)$ be the set of smooth vector fields on $\bS^2$.
	\begin{defn}\label{def:drift}
		We say that a smooth vector field $X \in \fX(\bS^2)$ is a \textit{longitudinal drift} if and only if:
		\begin{itemize}
			\item There exists an orthonormal laevorotatory coordinate system $\{O,\{\mathbf{e}_1,\mathbf{e}_2,\mathbf{e}_3\}\}$ on the Euclidean space $E^3$ such that $\bS^2=\{(x_1,x_2,x_3) \in E^3, x_1^2+x_2^2+x_3^2=1\}$;
			\item For any atlas $\cA$ on $(\bS^2,g)$, consider 
			\begin{equation}\label{eq:U1pm}
				U_\pm=\{(x_1,x_2,x_3)\in \bS^2: \ \pm x_1>0\}
			\end{equation}
		and the functions $\phi_{+}:U_\pm \to (0,\pi)\times (0,\pi)$ and $\phi_{-}:U_\pm \to (0,\pi)\times (-\pi,0)$ such that
			\begin{equation*}
				\phi_{\pm}^{-1}(\theta,\varphi)=(\sin(\theta)\cos(\varphi),\sin(\theta)\sin(\varphi),\cos(\theta)).
			\end{equation*}
		Then $\cA^{\prime}=\cA \cup\{(U_{\pm},\phi_\pm)\}$ is still an atlas on $(\bS^2,g)$ and there exists $\mu \in \R$ with $\mu \not = 0$ such that on $U_{\pm}$ it holds
		\begin{equation}\label{eq:longdrift}
			X_{|U_{\pm}}=\mu \pd{}{\varphi_1}.
		\end{equation}
		\end{itemize}
	In such a case, we say that $X$ is a longitudinal drift with intensity $\mu$ with respect to the orthonormal coordinate system $\{O,(\mathbf{e}_1,\mathbf{e}_2,\mathbf{e}_3)\}$.
\end{defn}
Clearly, one could ask whether a longitudinal drift actually exists. This question is addressed in the following proposition.
	\begin{prop}\label{prop:existlongdrift}
		For any $\mu \not = 0$ and any orthonormal laevorotatory coordinate system $\{O,(\mathbf{e}_1,\mathbf{e}_2,\mathbf{e}_3)\}$ there exist a unique longitudinal drift $X \in \fX(\bS^2)$ with respect to such a coordinate system and intensity $\mu$.
	\end{prop}
	\begin{proof}
		Let $\mu \not = 0$ and $\{O,(\mathbf{e}_1,\mathbf{e}_2,\mathbf{e}_3)\}$ be an orthonormal laevorotatory coordinate system on $E^3$. It is clear that if a longitudinal drift, with respect to such a coordinate system and with intensity $\mu$, exists, Equation \eqref{eq:longdrift} must hold true. To prove the statement, we just need to exhibit an atlas $\cA$ on $(\bS^2,g)$ and to show that, applying the appropriate change of variables formula, Equation \eqref{eq:longdrift} can be extended to a global vector field on $(\bS^2,g)$ in a unique way. To define such an atlas on $(\bS^2,g)$, we set
	\begin{equation*}
		U^i_{\pm}=\{(x_1,x_2,x_3)\in \bS^2: \ \pm x_i>0\}.
	\end{equation*} 
	Clearly, $U^1_{\pm}=U_{\pm}$ as in Equation \eqref{eq:U1pm}. Let $B_1=\{(x,y) \in \R^2: \ x^2+y^2<1\}$. Then we can define $\phi^i_{\pm}:U^i_{\pm} \to B_1$ by setting
	\begin{align*}
		(\phi^1_{\pm})^{-1}(x,y)&=(\pm \sqrt{1-x^2-y^2},x,y)\\
		(\phi^2_{\pm})^{-1}(x,y)&=(x,\pm \sqrt{1-x^2-y^2},y)\\
		(\phi^3_{\pm})^{-1}(x,y)&=(x,y,\pm \sqrt{1-x^2-y^2})
	\end{align*}
	It is easy to show that the atlas $\cA=\{(U^i_{\pm},\phi^i_{\pm}),i=1,2,3\}$ satisfies the assumptions in Definition \ref{def:drift}, i.e., that $\cA^\prime=\cA \cup \{(U_{\pm}, \phi_{\pm})\}$ is still an atlas. If we have a global vector field $X\in\mathfrak{X}(\bS^2)$ that satisfies  $(U_{\pm}, \phi_{\pm})$ we know that any global vector field $X\in\mathfrak{X}(\bS^2)$ satisfying Definition \ref{def:drift}, then it must be defined locally on $(U_{\pm}, \phi_{\pm})$ by \eqref{eq:longdrift} and must satisfy the usual change of variables formulae on the intersection of any pair of charts. In this way we can provide a construction of such a vector field $X$ by first determining its behaviour on $U^i_{\pm} \cap U_{\pm}$ and then if it is possible extending the local coordinates expression to the whole chart $(U^i_{\pm},\phi^i_{\pm})$. For instance, it is not difficult to check that on $U^1_{+}$ we have $X_{|{U^1_+}}=-\mu\sqrt{1-y_1^2}\pd{}{x_1}$, where $(x_1,y_1)$ are the local coordinates of the chart $\left(U^1_{+}, \phi^1_+\right)$. The remainder of the proof is constituted of easy but tedious computations and thus is omitted.
	\end{proof}
\begin{rmk}
	As it is evident from the local expression on the charts $U^1_{\pm}$, $X$ extends to the north and south poles (on which the spherical coordinates are actually ill-defined) as the null tangent vector (we denote it as $X_{|N}=0$). Let us also stress that $\bS^2$ is not parallelizable (as a consequence of the hairy ball theorem, see \cite[Problem 16.6]{lee2013introduction}), thus there is no way to write $X=f_1\pd{}{x_1}+f_2\pd{}{x_2}$ for two independent vector fields $\pd{}{x_i}$ and two smooth functions $f_i:\bS^2 \to \C$. 
\end{rmk}
 From now on, let us fix $\{O,(\mathbf{e}_1,\mathbf{e}_2,\mathbf{e}_3)\}$, an orthonormal laevorotatory coordinate system on $E^3$ and its associated spherical coordinate system. Moreover, since the coordinate system is fixed, we can refer directly to $\R^3$ in place of $E^3$. Furthermore, we will consider only longitudinal drifts $X \in \fX(\bS^2)$ associated to such a coordinate system, hence we can denote the longitudinal drift of intensity $\mu \not = 0$ as $X_\mu$ while we set $X_0=0$.
 
Let us recall that a function $f \in C^k(\bS^2)$ (possibly $k=\infty$) if its coordinate representation belongs to $C^k(\R^2)$ (see \cite[Chapter 2]{lee2013introduction}). However, in the specific case of the sphere, one can study the regularity of a function $f:\bS^2 \to \C$ by means of its $0$-homogeneous extension $f_0:\R^3\setminus \{0\} \to \C$, defined as
\begin{equation*}
	f_0(x)=f\left(\frac{x}{|x|}\right), \ \forall x \in \R^3 \setminus \{0\}.
\end{equation*}
Indeed (see \cite[Proposition 1.29]{rubin2015introduction}) $f \in C^k(\bS^2)$ if and only if $f_0 \in C^k(\R^3 \setminus \{0\})$. We can further define the gradient $\nabla_g$ on $(\bS^2,g)$ thanks to this identification. Indeed, if we denote by $\nabla$ the gradient on $\R^3$, it is clear that we can decompose, for $x \in \bS^2$ and $f \in C^1(\bS^2)$,
\begin{equation*}
	\nabla f_0(x)=\langle f_0,x\rangle x+\nabla_\tau f_0(x),
\end{equation*}
where $\nabla_\tau f_0(x)$ belongs to the tangent space $T_x\bS^2$. Moreover, since $f_0$ is $0$-homogeneous, it holds $\langle f_0,x\rangle=0$ and then
\begin{equation*}
	\nabla f_0(x)=\nabla_\tau f_0(x)=:\nabla_g f(x).
\end{equation*}
In spherical coordinates, the gradient $\nabla_g$ is represented as
\begin{equation}\label{eq:localgradient}
	\nabla_g=\begin{pmatrix} \pd{}{\theta} \\
		\frac{1}{\sin(\theta)}\pd{}{\varphi}
	\end{pmatrix}.
\end{equation}
The previous equality could be also obtained by the standard definition of gradient on a Riemannian manifold. Indeed, once we recall that, in spherical coordinates
\begin{equation}\label{eq:metriconsphere}
	g=\begin{pmatrix}
		1 & 0 \\
		0 & \sin^2(\theta)
	\end{pmatrix},
\end{equation}
then \eqref{eq:localgradient} follows from the relation $g(\nabla_g f,X)=X(f)$ for any $X \in \fX(\bS^2)$.

We denote the volume form on $(\bS^2,g)$ by $d\sigma$. Such a volume form coincide with the surface measure on $\bS^2$ once it is identified as a submanifold of $\R^3$. Furthermore, thanks to \eqref{eq:metriconsphere}, we know that, in spherical coordinates, it holds
\begin{equation*}
	d\sigma=\sin(\theta)d\theta d\varphi.
\end{equation*}
Once the volume form has been considered, one can define the divergence operator, whose action on a smooth vector field $X \in \fX(\bS^2)$ is denoted by $\nabla_g \cdot X$ (see, for instance, \cite[Exercise 16.11]{lee2013introduction} for a local representation of such an operator). The Laplace-Beltrami operator $\Delta_g$ on $(\bS^2,g)$ is then defined as $\Delta_g=\nabla_g \cdot \nabla_g f$ for any $f \in C^2(\bS^2)$. In spherical coordinates we have
\begin{equation}\label{eq:Laplacian}
	\Delta_g=\frac{1}{\sin(\theta)}\pd{}{\theta}\left(\sin(\theta)\pd{}{\theta}\right)+\frac{1}{\sin^2(\theta)}\pdsup{}{\varphi}{2}.
\end{equation}
As for the gradient, such an operator can be evaluated by considering the $0$-holomorphic extensions. Indeed, it holds $\Delta_g f\left(\displaystyle\frac{x}{|x|}\right)=|x|^2\Delta f_0(x)$ for any $x \in \R^3 \setminus \{0\}$.

It is well-known that $\Delta_g$ admits eigenvalues $\lambda_\ell=-\ell(\ell+1)$, $\ell=0,1,\dots$, with eigenspaces $\Lambda_\ell$ such that
\begin{equation*}
	\dim(\Lambda_\ell)=2\ell+1.
\end{equation*}
Any element of $\Lambda_\ell$ is called a $\ell$-th spherical harmonic. Among them, we can cite Laplace's spherical harmonics (see \cite[Section 3.4]{marinucci2011random}) 
\begin{equation}\label{eq:Ylm}
	Y_{\ell,m}(\theta,\varphi)=\sqrt{\frac{(2\ell+1)(\ell-m)!}{4\pi (\ell+m)!}}P_{\ell}^m(\cos(\theta))e^{im\varphi}, \ \theta \in [0,\pi], \ \varphi \in [0,2\pi),
\end{equation}
where $i$ is the imaginary unit of the complex plane $\C$, $\ell=0,1,\dots$, $m=-\ell,\dots,\ell$ and $P_{\ell}^m$ are the associated Legendre functions (see \cite[Section 4.7]{szeg1939orthogonal}), defined by the Rodriguez formula
\begin{equation*}
	P_\ell^m(s)=\frac{(-1)^m}{2^\ell \ell!}(1-s^2)^{\frac{m}{2}}\dersup{}{s}{\ell+m}(s^2-1)^{\ell}, \ s \in [0,1].
\end{equation*}
Let us recall, in particular, that
\begin{equation}\label{eq:Plm1}
	P_\ell^m(\pm 1)=\begin{cases} 0 & m \not = 0 \\ (\pm 1)^\ell & m=0. \end{cases}
\end{equation}
In case $m=0$, $P_\ell^0=:P_\ell$ are the Legendre polynomials.
Clearly, each choice of spherical coordinate system on $\bS^2$ leads to a different set of Laplace's spherical harmonics. We will refer to $Y_{\ell,m}$ just as spherical harmonics. These functions are orthonormal with respect to the volume form, i.e.
\begin{equation*}
	\int_{\bS^2}Y_{\ell_1,m_1}\overline{Y}_{\ell_2,m_2}d\sigma=\delta_{\ell_1,\ell_2}\delta_{m_1,m_2},
\end{equation*}
where 
\begin{equation*}
	\delta_{\ell_1,\ell_2}=\begin{cases} 1 & \ell_1=\ell_2 \\ 0 &  \ell_1 \not = \ell_2 \end{cases}
\end{equation*}
and, for any complex number $z \in \C$, $\overline{z} \in \C$ is its conjugate. In particular, for any $\ell=0,1,2,\dots$, the set $\{Y_{\ell,m}: \ m=-\ell,\dots,\ell\}$ is an orthonormal basis of $\Lambda_\ell$. Let us denote by $L^p(\bS^2)$ the space of functions $f:\bS^2 \to \C$ such that
\begin{equation*}
	\Norm{f}{L^p(\bS^2)}^p:=\int_{\bS^2}|f|^pd\sigma<+\infty,
\end{equation*}
where we say that $f=g$ if and only if $f(x)=g(x)$ for $\sigma$-almost any $x \in \bS^2$. It is well known that $\{Y_{\ell,m}: \ \ell=0,1,\dots, \ m=-\ell, \dots, \ell\}$ is an orthonormal basis of $L^2(\bS^2)$ (see \cite[(3.39)]{marinucci2011random}) and, for such a reason, any function $f \in L^2(\bS^2)$ can be represented via its Fourier-Laplace series, as
\begin{equation*}
	f=\sum_{\ell=0}^{+\infty}\sum_{m=-\ell}^{\ell}f_{\ell,m}Y_{\ell,m},
\end{equation*}
where
\begin{equation*}
	f_{\ell,m}=\int_{\bS^2} f \overline{Y}_{\ell,m}d\sigma
\end{equation*}
are called the Fourier-Laplace coefficients and the convergence holds in $L^2(\bS^2)$. Moreover, by Parseval's identity 
\begin{equation*}
	\Norm{f}{L^2(\bS^2)}=\sum_{\ell=0}^{+\infty}\sum_{m=-\ell}^{\ell}f^2_{\ell,m}.
\end{equation*}

Let us also recall the following addition formula (see \cite[(3.42)]{marinucci2011random}):
\begin{equation}\label{eq:additionthm}
	\sum_{m=-\ell}^{\ell}Y_{\ell,m}(x)\overline{Y}_{\ell,m}(y)=\frac{2\ell+1}{4\pi}P_\ell( x\cdot y), \quad \forall x,y \in \bS^2,
\end{equation}
where $x\cdot y$ is the scalar product of $x,y \in \R^3$. 
As a direct consequence, recalling that $P_\ell(1)=1$, we have
\begin{equation}\label{eq:additionthmspecial}
	\sum_{m=-\ell}^{\ell}|Y_{\ell,m}(x)|^2=\frac{2\ell+1}{4\pi}, \ \forall x \in \bS^2,
\end{equation}
which also implies that
\begin{equation}\label{eq:uniformbound}
	|Y_{\ell,m}(x)|\le \sqrt{\frac{2\ell+1}{4\pi}}, \ \forall x \in \bS^2.
\end{equation}
Furthermore it is clear that by definition of $Y_{\ell,m}$ we have $e^{i\mu t}Y_{\ell,m}(x)=Y_{\ell,m}(x_t^\mu)$, where $x=(\theta,\varphi)$ and $x_t^\mu=(\theta,\varphi+\mu t)$, and then
\begin{equation}\label{eq:additiondrift}
\sum_{m=-\ell}^{\ell}e^{i\mu t} Y_{\ell,m}(x)\overline{Y}_{\ell,m}(y)=\frac{2\ell+1}{4\pi}P_\ell( x_t^\mu\cdot y)=\frac{2\ell+1}{4\pi}P_\ell( x\cdot y_t^{-\mu}), \quad \forall x,y \in \bS^2.
\end{equation}
Now, let us consider $\mu \in \R$. We are interested in the elliptic operator
\begin{equation}\label{eq:generator}
	\cG_\mu=\Delta_{g}+X_\mu,
\end{equation}
which is diagonalizable in $L^2(\bS^2)$, as it can be shown by direct computation, and it is recalled in the following theorem.
\begin{thm}\label{thm:specG}
	$\cG_\mu$ admits purely discrete spectrum, with simple eigenvalues given by
	\begin{equation*}
		\lambda_{\ell,m}=i\mu m-\ell(\ell+1), \ \ell=0,1,\dots, \ m=-\ell,\dots,\ell.
	\end{equation*}
	In particular, it holds
	\begin{equation}\label{eq:eigenvalues}
		\cG_\mu Y_{\ell,m}=\lambda_{\ell,m}Y_{\ell,m}
	\end{equation}
	for any $\ell=0,1,\dots$ and $m=-\ell,\dots,\ell$.
\end{thm}

\subsubsection{Sobolev spaces on $\bS^2$}
Let us assume that $f \in C^2(\bS^2)$ with Fourier-Laplace coefficients $(f_{\ell,m})_{\substack{\ell=0,1,\dots \\ m=-\ell,\dots,\ell}}$. Then it is clear that $\Delta_g f \in L^2(\bS^2)$ is well defined and
\begin{equation}\label{eq:LB}
	\Delta_g f=-\sum_{\ell=0}^{+\infty}\sum_{m=-\ell}^{\ell}f_{\ell,m}\ell(\ell+1)Y_{\ell,m}.
\end{equation}
Recalling that $\overline{\Delta_g f}=\Delta_g \overline{f}$ for any $f \in C^2(\bS^2)$, as a consequence of Stokes theorem, we have the following Green identity for any functions $f,h \in C^2(\bS^2)$
\begin{equation}\label{eq:selfadjoint}
	\int_{\bS^2} \overline{h}\Delta_g f d\sigma=-\int_{\bS^2}g(\nabla_g f, \nabla_g \overline{h})d\sigma=\int_{\bS^2} \Delta_g \overline{h} f d\sigma.
\end{equation}
This identity implies that $\Delta_g$ is a self-adjoint operator. Furthermore, if $f \in C^1(\bS^2)$ and $X \in \fX(\bS^2)$, then the following integration by parts formula holds
\begin{equation*}
	\int_{\bS^2}g(\nabla_g f, X)d\sigma=-\int_{\bS^2} f \nabla_g \cdot X d\sigma.
\end{equation*}
Thanks to the previous relations we can define the weak gradient and the weak Laplacian of a function. Indeed, we say that $f \in H^1(\bS^2)$ if and only if $f \in L^2(\bS^2)$ and there exists a vector fied $Y$ such that for any $X \in \fX(\bS^2)$ it holds
\begin{equation*}
	\int_{\bS^2}g(Y, X)d\sigma=-\int_{\bS^2} f \nabla_g \cdot X d\sigma.
\end{equation*}
and $|Y|:=\sqrt{g(Y,Y)}$ belongs to $L^2(\bS^2)$. In such a case, we denote $Y=\nabla_g f$ and we recall that if $f \in C^1(\bS^2)$ these definitions coincide. Furthermore, $f \in H^2(\bS^2)$ if and only if $f \in H^1$ and there exists a function $\widetilde{f} \in L^2$ such that for any $h \in C^2(\bS^2)$ it holds
\begin{equation*}
	\int_{\bS^2} \overline{h}\widetilde{f} d\sigma=-\int_{\bS^2}g(\nabla_g f, \nabla_g \overline{h})d\sigma.
\end{equation*}
In such a case we denote $\widetilde{f}=\Delta_g f$ and if $f \in C^2(\bS^2)$ the two definitions coincide. The spaces $H^1(\bS^2)$ and $H^2(\bS^2)$ are Hilbert spaces once equipped with the norms
\begin{align*}
	\Norm{f}{H^1(\bS^2)}&:=\Norm{f}{L^2(\bS^2)}+\Norm{|\nabla_g f|}{L^2(\bS^2)}\\
	\Norm{f}{H^2(\bS^2)}&:=\Norm{f}{L^2(\bS^2)}+\Norm{|\nabla_g f|}{L^2(\bS^2)}+\Norm{\Delta_g f}{L^2(\bS^2)}.
\end{align*}
Furthermore $H^1(\bS^2)\subset L^2(\bS^2)$. If we denote by $H^{-1}(\bS^2)$ the dual of $H^1(\bS^2)$, then we obtain the Gelfand triple
\begin{equation*}
	H^1(\bS^2)\subset L^2(\bS^2) \subset H^{-1}(\bS^2).
\end{equation*}
Let us recall that $H^1(\bS^2)$ is isomorphic to $H^{-1}(\bS^2)$ by the Riesz-Fisher theorem, but they do not coincide. In any case, for any function $f \in H^1(\bS^2)$, one can clearly define the functional
\begin{equation*}
	\Delta_g f: h \in H^1(\bS^2) \mapsto -\int_{\bS^2}g(\nabla_g f,\nabla_g \overline{h})d\sigma \in \C 
\end{equation*}
which is clearly linear and bounded in $H^1(\bS^2)$ and thus belongs to $H^{-1}(\bS^2)$. In this way we are able to extend the definition of the Laplace-Beltrami operator to functions that only belong to $H^1(\bS^2)$. The action of $\Delta_g f$ on $h$ will be still denoted through integrals. Furthermore, by definition, if $f,h \in H^1(\bS^2)$, then \eqref{eq:selfadjoint} is still valid. This means that $\Delta_g$ admits a self-adjoint extension to $H^1(\bS^2)$.

Moreover, let us observe that if $f \in H^2(\bS^2)$, then \eqref{eq:LB} still holds. It actually holds even in $H^1(\bS^2)$, where the convergence of the series has to be intended in $H^{-1}(\bS^2)$. This tells us, further, that if $f \in H^2(\bS^2)$ then
\begin{equation}\label{eq:H2cond}
	\sum_{\ell=0}^{+\infty}\ell^2(\ell+1)^2\sum_{m=-\ell}^{\ell}|f_{\ell,m}|^2<+\infty.
\end{equation}
In particular, let us define
\begin{equation*}
	A_\ell(f):=\sum_{m=-\ell}^{\ell}|f_{\ell,m}|^2,
\end{equation*}
where the sequence $(A_\ell(f))_{\ell \ge 0}$ is called the power-spectrum of $f$ so that we can rewrite \eqref{eq:H2cond} as
\begin{equation}\label{eq:H2cond2}
	\sum_{\ell=0}^{+\infty}\ell^4A_\ell(f)<+\infty.
\end{equation}
Actually, condition \eqref{eq:H2cond2} is equivalent to $f \in H^2(\bS^2)$. To better understand this, let us recall the following result (see \cite[Theorems A.46 and A.47]{rubin2015introduction}).
\begin{thm}\label{thm:smooth}
	Let $f \in L^2(\bS^2)$ with Laplace-Fourier coefficients $(f_{\ell,m})_{\substack{\ell=0,1,\dots \\ m=-\ell,\dots,\ell}}$.
	\begin{itemize}
		\item[(i)] If $f \in C^{2r}(\bS^{2})$ for some $r>1$, then
		\begin{equation}\label{eq:sumcond}
			\sum_{\ell=0}^{+\infty}\ell^{k}\sum_{m=\ell}^{\ell}|f_{\ell,m}|<+\infty
		\end{equation}
		for all $0 \le k<2r-2$;
		\item[(ii)] If $f \in C^{\infty}(\bS^{2})$, then the inequality \eqref{eq:sumcond} holds for all $k \ge 0$;
		\item[(iii)] If the inequality \eqref{eq:sumcond} holds for some $k \ge \frac{1}{2}$, then $f \in C^r(\bS^2)$ for any $r \le k-\frac{1}{2}$ and the series
		\begin{equation}\label{eq:seriesLF}
			f=\sum_{\ell=0}^{+\infty}\sum_{m=-\ell}^{\ell}f_{\ell,m}Y_{\ell,m}
		\end{equation}
		converges absolutely and uniformly;
		\item[(iv)] If the inequality \eqref{eq:sumcond} holds for any $k \ge 0$, then $f \in C^\infty(\bS^2)$ and the series \eqref{eq:seriesLF} converges absolutely and uniformly.
	\end{itemize}
\end{thm}
Next, let us introduce, in general, the following \textit{Sobolev spaces}: we say that $f \in H^s(\bS^2)$ for some $s \ge 0$ if and only if $f \in L^2(\bS^2)$ and its power spectrum satisfies
\begin{equation*}
	[f]_{H^s(\bS^2)}^2:=\sum_{\ell=0}^{+\infty}A_\ell(f)\ell^{2s}<+\infty.
\end{equation*}
Let us observe that if $s>0$, then $[\cdot]_{H^s(\bS^2)}^2$ is a seminorm on $H^s(\bS^2)$, as constant functions $f \in L^2(\bS^2)$ satisfy $[f]_{H^s(\bS^2)}^2=0$. However, a norm on such spaces can be defined by setting
\begin{equation*}
	\Norm{f}{H^s(\bS^2)}:=\Norm{f}{L^2(\bS^2)}+[f]_{H^s(\bS^2)}.
\end{equation*}
Thanks to Theorem \ref{thm:smooth}, we can prove that functions belonging to $H^s(\bS^2)$ are actually smoother than just in $L^2(\bS^2)$. To do this, let us prove the following technical lemma, that will be also used in the next sections.
\begin{lem}\label{lem:seriesconv}
	If $f \in H^s(\bS^2)$ for some $s \ge 0$, then for any $k<s-\frac{1}{2}$ it holds
	\begin{equation*}
		\sum_{\ell=0}^{+\infty}\ell^k \sqrt{A_\ell(f)}<+\infty.
	\end{equation*}
\end{lem}
\begin{proof}
	Let us denote by $\zeta$ the Riemann Zeta function and let us notice that, by Jensen's inequality and the fact that the square root function is concave,
	\begin{align*}
		\sum_{\ell=0}^{+\infty}\ell^k\sqrt{A_\ell(f)}&=\frac{1}{\zeta(2s-2k)}\sum_{\ell=0}^{+\infty}\ell^{2k-2s}\sqrt{\zeta^2(2s-2k)\ell^{4s-2k}A_\ell(f)}\\
		&\le \sqrt{\zeta(2s-2k)\sum_{\ell=0}^{+\infty}\ell^{2s}A_\ell(f)}<+\infty.
	\end{align*}
\end{proof}
As a direct consequence, we have the following result.
\begin{thm}\label{thm:smoothness}
	If $f \in H^s(\bS^2)$ for some $s>2$, then $f \in C^r(\bS^2)$ for any $r<s-1$.
\end{thm}
\begin{proof}
	Let $r<s-1$ and consider $k \ge 0$ such that $2r+2 \le 2k+1 < 2s$. By the equivalence of the norms in finite dimensional vector spaces, we clearly have
	\begin{equation*}
		\sum_{\ell=0}^{+\infty}\ell^k\sum_{m=\ell}^{\ell}|f_{\ell,m}| \le C\sum_{\ell=0}^{+\infty}\ell^k\sqrt{A_\ell(f)}<+\infty,
	\end{equation*}
	by Lemma \ref{lem:seriesconv}. Hence, by Theorem \ref{thm:smooth}, we know that $f \in C^r(\bS^2)$.
\end{proof}
It is clear that $H^0(\bS^2)=L^2(\bS^2)$. The fact that $H^1(\bS^2)$ coincides with the one defined before is less trivial to observe: we refer to \cite[(7.11)]{figalli2015isoperimetry} and references therein for such identity. Furthermore, it is clear that if $[f]_{H^s(\bS^2)}<+\infty$, then also $[f]_{H^{s'}(\bS^2)}<+\infty$ for any $0 \le s'<s$. Thanks to such an observation, we have that the space $H^2(\bS^2)$ coincides with the one defined before. 

\subsubsection{The Brownian motion with longitudinal drift on $\bS^2$}
Following the lines of \cite{wang2014analysis}, we now give the definition of Brownian motion on $\bS^2$ with longitudinal drift. It is clear that $\cG_\mu$ defined in Equation \eqref{eq:generator} is an elliptic operator which generates a diffusion process $W^\mu=\{W^\mu(t), \ t \ge 0\}$ on $\bS^2$ with respect to a suitable probability basis $(\Omega, \cF, \{\cF\}_{t \ge 0}, \bP)$, which from now on will be fixed. Furthermore, $\bS^2$ is a compact manifold without boundary, thus the geodesic distance is bounded and it is clear that $W^\mu(t)$ is non-explosive. Hence, the following proposition holds.
\begin{prop}
	For any $\mu \in \R$, the elliptic operator $\cG_\mu$ defined in Equation \eqref{eq:generator} generates a non explosive diffusion process $W^\mu$.
\end{prop}
	In general, given a suitable vector field $Z \in \fX(\bS^2)$, one can define a Brownian motion on $\bS^2$ with drift $Z$ as the diffusion process on $\bS^2$ generated by $\Delta_g+Z$. This can be also found by solving the stochastic differential equation
	\begin{equation*}
		dW^Z(t)=dW(t)+Z(W^\mu(t))dt
	\end{equation*}
	where $W(t):=W^0(t)$ is a Brownian motion on $\bS^2$ and its differential is defined by means of horizontal lift (see \cite[Chapter $2$]{wang2014analysis}). Since $W_Z(t)$ is a diffusion process, then it is clear that it holds $\bP(W_t=(0,0,1))=\bP(W_t=(0,0,-1))=0$. With this observation, in \cite[Section $7$]{LGY:86} the authors expressed $W_Z(t)=(\theta_Z(t),\varphi_Z(t))$ in spherical coordinates, where $\theta_Z(t)$ and $\varphi_Z(t)$ are solutions of the stochastic differential equations
	\begin{align*}
		d\theta^Z(t)&=dw^{(1)}(t)+\left(\frac{1}{2}\cot(\theta^Z(t))+Z_1(\theta^Z(t),\varphi^Z(t))\right)dt \\
		d\varphi^Z(t)&=\frac{1}{\sin(\theta^Z(t))}\left(dw^{(2)}(t)+Z_2(\theta^Z(t),\varphi^Z(t))dt\right),
	\end{align*}
where $Z_i$, $i=1,2$, are such that $Z=Z_1\nabla_g^1+Z_2\nabla_g^2$. In our specific case, we have $Z_1 \equiv 0$ and $Z_2=\mu\sin(\theta)$. Clearly, for $\mu=0$ we recover the Brownian motion on $\bS^2$ as defined in \cite{perrin1928etude,yosida1949brownian}. Furthermore, as $\mu=0$, the equations of the spherical coordinates of a Brownian motion of $\bS^2$ are given by
\begin{align*}
	d\theta(t)&=dw^{(1)}(t)+\frac{1}{2}\cot(\theta(t))dt \\
	d\varphi(t)&=\frac{1}{\sin(\theta(t))}dw^{(2)}(t).
\end{align*}
Let us denote $W(t)=(\theta(t),\varphi(t))$. Once we consider $\mu \not = 0$, since $Z_1 \equiv 0$ and $Z_2(\theta)=\mu\sin(\theta)$, the equations become
\begin{align*}
	d\theta^\mu(t)&=dw^{(1)}(t)+\frac{1}{2}\cot(\theta^\mu(t))dt \\
	d\varphi^\mu(t)&=\frac{1}{\sin(\theta^\mu(t))}dw^{(2)}(t)+\mu dt,
\end{align*}
hence it is clear that $\varphi^\mu(t)=\varphi(t)+\mu t$ and $\theta^\mu(t)=\theta(t)$. Thus we have in general
\begin{equation}\label{eq:sphericalcoord}
	W^\mu(t)=(\theta(t),\varphi(t)+\mu t).
\end{equation}
For any $x \in \bS^2$, we denote $\bP_x(\cdot):=\bP(\cdot|W^\mu(0)=x)$. We can use Theorem \ref{thm:specG}, to provide a spectral decomposition of the probability density function of $W^\mu(t)$. Actually, one can first prove the following result.
\begin{thm}\label{eq:sKolmogorov}
	Let $f \in L^2(\bS^2)$ with Fourier-Laplace coefficients $(f_{\ell,m})_{\substack{\ell=0,1,\dots \\ m=-\ell,\dots,\ell}}$. Then the Cauchy problem
	\begin{equation}\label{eq:sBK}
		\begin{cases}
			\displaystyle \pd{u}{t}(t)=\cG_\mu u(t) & t>0 \\
			u(0)=f
		\end{cases}
	\end{equation}
	admits a unique strong solution $u \in C(\R_0^+; L^2(\bS^2)) \cap C^1(\R^+;L^2(\bS^2))$, such that
	\begin{equation}\label{eq:seru}
		u(t)=\sum_{\ell=0}^{+\infty}\sum_{m=-\ell}^{\ell}f_{\ell,m} e^{(i\mu m-\ell(\ell+1))t}Y_{\ell,m}.
	\end{equation}
	Furthermore, $u \in C^\infty(\R^+ \times \bS^2)$. 
\end{thm}
The proof relies on standard arguments which are analogous to the ones of Theorem \ref{thm:specdecKolm} and thus it is omitted.
\begin{rmk}\label{rmk:Sobolev}
If the series $f=\sum_{\ell=0}^{+\infty}\sum_{m=-\ell}^{\ell}f_{\ell,m}Y_{\ell,m}$ is uniformly convergent, then, by Abel's uniform convergence test, 
the series \eqref{eq:seru} is uniformly convergent for $t \in \R_0^+$, $\theta \in [0,\pi]$ and $\varphi \in [0,2\pi)$. This is, for instance, guaranteed if $f \in H^1(\bS^2)$ and $|\nabla_g f|\le C$ $\sigma$-almost everywhere on $\bS^2$ for some constant $C>0$, as implied by \cite[Theorem 5]{ragozin1971uniform}. Furthermore, the fact that $u(t) \in C^\infty(\bS^2)$ despite $f \in L^2(\bS^2)$ is a peculiar property of parabolic equations, usually called \textit{parabolic smoothing effect}. This can be used to define integral kernels to achieve smooth approximations of non smooth functions (see, for instance, \cite[Example 1.18]{rubin2015introduction}).
\end{rmk}

As a direct consequence of Theorem \ref{eq:sKolmogorov}, we get the following corollary, whose proof is postponed to Appendix \ref{appCOR}, since it is based on standard arguments. 
\begin{cor}\label{cor:transition_density_standard}
	There exists a function $p_\mu:\R_0^+ \times \bS^2 \times \bS^2 \to \R$ such that for any $x \in \bS^2$, $t > 0$ and any $A \in \cB(\bS^2)$, where $\cB(\bS^2)$ is the Borel $\sigma$-algebra of $\bS^2$, it holds
	\begin{equation}\label{eq:density}
		\bP_x(W^\mu(t) \in A)=\int_{A} p_\mu(t,y|x)d\sigma(y).
	\end{equation}
In particular, for any $t>0$, $x,y \in \bS^2$, we have
\begin{equation}\label{eq:specdensity}
	p_\mu(t,y|x)=\sum_{\ell=0}^{+\infty}\sum_{m=-\ell}^{\ell}e^{(i\mu m-\ell(\ell+1))t}Y_{\ell,m}(x)\overline{Y}_{\ell,m}(y).
\end{equation}
Furthermore, for any $t>0$ and $x,y \in \bS^2$, it holds
\begin{equation}\label{eq:adjointdensity}
	p_\mu(t,y|x)=p_{-\mu}(t,x|y)=p(t,y_t^{-\mu}|x)=p(t,x_t^{\mu}|y)
\end{equation}
where $p(t,y|x):=p_0(t,y|x)$.
\end{cor}
Thanks to such a spectral decomposition, we are able to identify the stationary distributions of $W^\mu(t)$. 
To do this, let us denote by $\cP(\bS^2)$ the space of Borel probability measures on $\bS^2$
and for any $A \in \cF$, let $\bP_{\mathfrak{P}}(A)=\bP(A|W^\mu(0)=X)$, for some random variable $X$ with law $\mathfrak{P}$.
Next proposition  
identifies the stationary distribution of $W^\mu$.
\begin{prop}\label{prop:statd}
		Let $\mathfrak{P} \in \cP(\bS^2)$ 
	and denote $\mathfrak{P}_U=\frac{\sigma}{4\pi}$. Then, for any $t>0$ and any $A \in \cB(\bS^2)$ it holds
	\begin{equation*}
		\lim_{s\to\infty}\bP_{\mathfrak{P}}(W^\mu(s) \in A)=\frac{\sigma(A)}{4\pi}=\bP_{\mathfrak{P}_U}(W^\mu(t) \in A)\,.
	\end{equation*}
\end{prop}
The proof of this statement can be obtained in the same way as the one of Proposition \ref{prop:stat} and thus it is omitted.


Since $W^\mu$ is a Feller process, we can consider the Feller semigroup $T_t:C(\bS^2) \to C(\bS^2)$ defined as
\begin{equation}\label{eq:Markovsemi}
T_tf(x)=\E_x[f(W^\mu(t))], \ \forall t \ge 0 \mbox{ and  }   \forall x \in \bS^2.
\end{equation}
In the following lemma, whose proof is given in Appendix \ref{app:lemext}, we show that we can extend such a semigroup to $L^p(\bS^2)$, for any $p\ge1$.

\begin{lem}\label{lem:extendLp}
	For any $1 \le p <\infty$, the Feller semigroup $T_t:C(\bS^2) \to C(\bS^2)$ defined as in \eqref{eq:Markovsemi} can be extended to a strongly continuous, positivity preserving sub-Markov non-expansive semigroup (that we still denote $\{T_t\}_{t \ge 0}$) with $T_t:L^p(\bS^2) \to L^p(\bS^2)$ and it holds
	\begin{equation*}
		T_tf(x)=\E_x[f(W^\mu(t))], \ \forall t \ge 0 \mbox{ and for almost any }x \in \bS^2.
	\end{equation*}
\end{lem}

It is clear that thanks to the full knowledge of the spectrum of the generator $\cG_\mu$, we could construct \textit{by hand} the solution of the associated Kolmogorov equation, in terms of a Fourier-Laplace series. Furthermore, this allowed us to carry on an extensive study on stationary and limit distributions of the drifted Brownian motion $W^\mu$ generated by $\cG_\mu$.

Our aim is to extend such an approach to the time-fractional and generalized time-fractional case upon substituting the exponential with a suitable eigenfunction of the involved nonlocal operators. This is the content of the next sections.

\subsection{Subordinators, Inverse subordinators and Caputo-type nonlocal operators}\label{subs:subecc}

\subsubsection{Subordinators and inverse subordinators}\label{subs:Sub}
We recall here some basic definitions on subordinators and their inverses. We refer to \cite{bertoin1999subordinators} for further informations.
\begin{defn}\label{defn:subordinator}
	A subordinator $\{S(t), \ t \ge 0\}$ on $(\Omega, \cF, \{\cF_t\}_{t \ge 0}, \bP)$ is a real increasing L\'evy process such that $S(0)=0$ almost surely. In particular, for any $t>0$, $S(t) \ge 0$ almost surely. Given a subordinator $S(t)$ we define its inverse as
	\begin{equation*}
		L(t)=\inf\{s>0: \ S(s)>t\}.
	\end{equation*}
\end{defn}
Clearly, a subordinator is uniquely determined (in law) by its Laplace exponent $\Phi$, i.e. a function $\Phi:\R_0^+ \to \R$ such that
\begin{equation}\label{eq:subL}
	\E[e^{-\lambda S(t)}]=e^{-t\Phi(\lambda)}, \ \forall t\ge 0, \ \forall \lambda \ge 0.
\end{equation}
For such a reason, we denote by $S_\Phi(t)$ any subordinator that admits $\Phi$ as Laplace exponent and $L_{\Phi}$ its inverse. 
We will assume throughout the paper that 
\begin{equation}\label{eq:Bernstein}
	\Phi(\lambda)=\int_0^{+\infty}(1-e^{-\lambda t})\nu_\Phi(dt), \ \forall \lambda \ge 0.
\end{equation}
where 
\begin{equation*}
\int_0^{+\infty} (1\wedge t)\nu_\Phi(dt)<+\infty \quad \text{and} \quad \nu_\Phi(0,+\infty)=+\infty\,,
\end{equation*}
that is $\Phi$ is a Bernstein function with zero drift and killing coefficients, and infinite activity, see \cite[Theorem 3.2]{schilling2012bernstein}.
Any function that can be expressed as in \eqref{eq:Bernstein} is the Laplace exponent of a subordinator. 

Under this assumption, $S_\Phi$ is strictly increasing, thus $L_\Phi$ is continuous, and for any $t>0$ the random variable $L_\Phi(t)$ admits density
\begin{equation}\label{eq:density2}
	f_L(s;t)=\int_0^{t}\overline{\nu}_\Phi(t-\tau)f_{S}(d\tau;s), \ s > 0,
\end{equation}
where $f_{S}(d\tau;s)$ is the probability law of $S_\Phi(s)$ and $\overline{\nu}_\Phi(t)={\nu}_\Phi(t,+\infty)$. Examples of functions $\Phi$ with these properties are:
$$
\Phi(\lambda)=\lambda^\alpha, \quad \Phi(\lambda)=(\lambda+\theta)^\alpha-\theta^\alpha, \quad \Phi(\lambda)=\log(1+\lambda), \quad \Phi(\lambda)=\log(1+\lambda^\alpha),	
 $$
 where $\alpha\in(0,1)$ and $\theta>0$. For further examples we refer to \cite{schilling2012bernstein}. 
Integrating by parts in \eqref{eq:Bernstein}, we get
\begin{equation}\label{eq:Laplacenu}
	\int_0^\infty e^{-\lambda t} \overline{\nu}_\Phi(t)  \,dt=\frac{\Phi(\lambda)}{\lambda}, \ \Re(\lambda)>0.
\end{equation}
Combining this with \eqref{eq:density2} and \eqref{eq:subL}, we have the Laplace transform
\begin{equation}\label{eq:Laplace}
	\int_0^\infty e^{-\lambda t} f_L(s;t)  \,dt=\frac{\Phi(\lambda)}{\lambda}e^{-s\Phi(\lambda)}, \ \Re(\lambda)>0.
\end{equation}
There is no closed formula for the Laplace transform in the variable $s$. However, under our assumptions it has been shown in \cite[Lemma 4.1]{ascione2021abstract} that such a Laplace transform exists for any $\lambda \in \C$. We denote
\begin{equation}\label{eq:Phexp}
	\fe_\Phi(t;\lambda)=\E[e^{\lambda L_\Phi(t)}], \ \forall \lambda \in \C.
\end{equation}
For such a function, we can provide the following upper bound.
\begin{lem}
	For any $t_0>0$ there exists a constant $K>0$ such that for any $\lambda \in \C$ with $\Re(\lambda)<0$ it holds
	\begin{equation}\label{eq:ineq1}
		-\Re(\lambda)|\fe_\Phi(t;\lambda)|\le K, \ \forall t \ge t_0.
	\end{equation}
\end{lem}
\begin{proof}
	Let us stress that if $\lambda \in \R$ with $\lambda<0$, then inequality \eqref{eq:ineq1} has been proven in \cite[Proposition $4$]{ascione2022non2}. Now let $\lambda \in \C$ with $\Re(\lambda)<0$. Then it holds
	\begin{align*}
		-\Re(\lambda)|\fe_\Phi(t;\lambda)|=-\Re(\lambda)|\E[e^{\lambda L(t)}]|\le -\Re(\lambda)\E[|e^{\lambda L(t)}|]=-\Re(\lambda)\E[e^{\Re(\lambda) L(t)}]\le K,
	\end{align*}
where we used the fact that $\Re(\lambda)<0$.
\end{proof}
\begin{rmk}
	Inequality \eqref{eq:ineq1} is sharp. Indeed, if $\Phi(\lambda)=\lambda^\alpha$ with $\alpha\in(0,1)$, then $e_{\Phi}(t;\lambda)=E_\alpha(-\lambda t^{\alpha})$, where $E_\alpha(z)=\sum_{k=0}^{+\infty}\frac{z^k}{\Gamma(\alpha k+1)}$ is the Mittag-Leffler function \cite{bingham1971limit}. In particular, $\lim_{\lambda\to\infty}\lambda E_\alpha(-\lambda t^{\alpha})=t^{-\alpha}/\Gamma(1-\alpha)$, as shown in \cite[Proposition 36]{gorenflo2020mittag}. 
\end{rmk}
We also recall the following asymptotic behaviour, that is shown in item (4) of \cite[Theorem 2.1]{meerschaert2019relaxation} under more restrictive assumptions, but whose proof is exactly the same under our milder hypotheses.
\begin{lem}\label{lem:asbeh}
	Assume that $\Phi(\lambda)=\lambda^\alpha \cL(\lambda)$, where $\alpha \in [0,1)$ and $\cL(\lambda)$ satisfies $\lim_{\lambda \to 0^+}\frac{\cL(C\lambda)}{\cL(\lambda)}=1$ for all $C \ge 1$. Then for all $\lambda>0$
	\begin{equation*}
		\lim_{t \to +\infty}\frac{\Gamma(1+\alpha)t^\alpha \lambda \fe_\Phi(t;-\lambda)}{\alpha \cL(1/t)}=1.
	\end{equation*}
\end{lem}

Let us stress that, despite $S_\Phi(t)$ is a Markov process, $L_\Phi(t)$ does not exhibit the Markov property, but just the semi-Markov one. Actually, if we consider a strongly Markov process $M(t)$ with values in a suitable Banach space and a subordinator $S_\Phi(t)$ independent of it, the process $M_\Phi(t):=M(L_\Phi(t))$ is a semi-Markov process, in the sense that the couple $(M(t),t-S_\Phi(L_\Phi(t)^{-}))$ is strongly Markov. This can be shown by observing that the couple $(M(t),S_\Phi(t))$ is a Markov additive processes and then using the theory of regenerative processes as in \cite{kaspi1988regenerative}. We refer to the process $M_\Phi(t)$ as a \textit{time-changed process} with \textit{parent process} $M(t)$ and \textit{time-change} $L_\Phi(t)$.

\subsubsection{Caputo-type nonlocal operators and their eigenfunctions}
In this section, we define the operators that concern our model.
\begin{defn}\label{defn:strongder}
	Let $\Phi$ be as in \eqref{eq:Bernstein} and let $(\mathfrak{B},\Norm{\cdot}{})$ be a Banach space. For any function $f:\R_0^+ \to \mathfrak{B}$ we define
	\begin{equation*}
		\dersup{f}{t}{\Phi}(t):=\der{}{t}\int_0^{t}\overline{\nu}_\Phi(t-\tau)(f(\tau)-f(0^+))d\tau, \ t>0,
	\end{equation*}
	where the integral is a Bochner integral and the derivative has to be intended as a strong limit in $\mathfrak{B}$, provided such a quantity exists.
\end{defn}
These operators have been introduced independently in \cite{kochubei2011general,toaldo2015convolution} and they have been widely studied in literature. In particular, if $f:\R_0^+ \to \mathfrak{B}$ is absolutely continuous and a.e.~differentiable, then $\dersup{f}{t}{\Phi}$ is well defined and it holds (see \cite[Lemma 1.1]{ascione2022non})
\begin{equation*}
	\dersup{f}{t}{\Phi}(t)=\int_0^{t}\overline{\nu}_\Phi(t-\tau)f'(\tau)d\tau, \ t >0.
\end{equation*}

Now let $\{T_t\}_{t \ge 0}$ be a strongly continuous Markov semigroup on $\mathfrak{B}$ and $(\mathcal{T},{\rm Dom}(\mathcal{T}))$ be its generator. Then, for any $f \in {\rm Dom}(\mathcal{T})$, one could study the abstract Cauchy problem
\begin{equation}\label{eq:ACP}
	\begin{cases}
		\displaystyle \dersup{u}{t}{\Phi}(t)=\mathcal{T}u(t) & t>0 \\
		u(0)=f.
	\end{cases}
\end{equation}
For these problems, the following subordination principle has been shown in \cite{toaldo2015convolution,chen2017time}.
\begin{thm}\label{thm:strongsolsemigroup}
The function
	\begin{equation*}
		f(t)=\E[T_{L_\Phi(t)}f_0]
	\end{equation*}
is the unique strong solution of \eqref{eq:ACP}, in the sense that
\begin{itemize}
	\item $u \in C(\R_0^+;\mathfrak{B})$;
	\item $\displaystyle\dersup{u(t)}{t}{\Phi}$ is well-defined in $\mathfrak{B}$ for any $t>0$;
	\item $\displaystyle\dersup{u}{t}{\Phi} \in C(\R^+; \mathfrak{B})$;
	\item $u$ satisfies the equalities in \eqref{eq:ACP}.
\end{itemize}
\end{thm}
\begin{rmk}
	The previous theorem actually generalizes \cite[Theorem 3.1]{baeumer2001stochastic}, since if $\Phi(\lambda)=\lambda^\alpha$, then $\dersup{f}{t}{\Phi}(t)$ is the well-known Caputo derivative. Furthermore, if $ \Phi(\lambda)=(\lambda+\theta)^\alpha-\theta^\alpha$, then $\dersup{f}{t}{\Phi}(t)$ is a tempered fractional derivative, as in \cite{meerschaert2019stochastic,sabzikar2015tempered}. For similar regularity results see also \cite[Theorem 2.1]{chen2017time} and \cite[Theorem 3]{ascione2021time} in the case of a complete Bernstein function.
\end{rmk}
Among the abstract Cauchy problems of the form \eqref{eq:ACP}, let us consider the \textit{growth equation}
\begin{equation}\label{eq:growth}
	\begin{cases}
		\displaystyle \dersup{f}{t}{\Phi}(t)=\lambda f(t) & t>0 \\
		f(0)=1,
	\end{cases}
\end{equation}
where $\lambda \in \C$ and $f:\R_0^+ \to \C$. Such an equation has been studied with different techniques in \cite{kochubei2019growth,kolokol2019mixed,meerschaert2019relaxation,ascione2021abstract} for $\lambda \in \R$. However, adopting the same proof as in \cite[Proposition 4.3]{ascione2021abstract}, one can show the following statement.
\begin{prop}\label{prop:exp}
	For any $\lambda \in \C$ the function $\fe_\Phi(\cdot;\lambda)$ is the unique solution of \eqref{eq:growth}.
\end{prop}

\section{Main results}\label{sec:MR}

\subsection{The time-nonlocal Kolmogorov equation: strong solutions}\label{sec:strongsol}
We are interested in the time-nonlocal Kolmogorov equation
\begin{equation}\label{eq:nonlocbKq}
	\begin{cases}
		\displaystyle \dersup{u}{t}{\Phi}(t)=\cG_\mu u(t) & t>0, \\
		u(0)=f,
	\end{cases}
\end{equation}
for suitable initial data $f$. It is clear, as we observed before, that $C^2(\bS^2) \subset {\rm Dom}(\cG_\mu)$, hence we are able to provide a strong solution of \eqref{eq:nonlocbKq}, for $f\in C^2(\bS^2)$, by means of Theorem \ref{thm:strongsolsemigroup}.
Actually, we can do more than this. Indeed, for less regular initial data, we can provide existence and uniqueness of the strong solutions of the Kolmogorov equation via a spectral decomposition result. The proof of such a theorem is similar to the one of \cite[Theorem 3.2]{d2014time} for the case $\mu=0$. However, by using the finer estimate provided by Lemma \ref{lem:seriesconv}, we are able to further weaken their regularity assumptions on the initial data.
\begin{thm}\label{thm:specdecKolm}
	Assume that $f \in H^s(\bS^2)$ with $s>1$ and
	\begin{equation*}
		f=\sum_{\ell=0}^{+\infty}\sum_{m=-\ell}^{\ell}f_{\ell,m}Y_{\ell,m}.
	\end{equation*}
Then the abstract Cauchy problem \eqref{eq:nonlocbKq} in $H^s(\bS^2)$ admits a unique strong solution $u$, in the sense that:
\begin{itemize}
	\item[(i)] $u \in C(\R_0^+;H^s(\bS^2))\cap C(\R^+;H^{s+2}(\bS^2))$;
	\item[(ii)] $\displaystyle \dersup{u}{t}{\Phi}(t)$ is well-defined in $H^{s}(\bS^2)$ for any $t>0$;
	\item[(iii)] $\displaystyle \dersup{u}{t}{\Phi} \in C(\R^+;H^{s}(\bS^2))$;
	\item[(iv)] $u$ satisfies the equalities in \eqref{eq:nonlocbKq}.
\end{itemize}
In particular, such a solution is given by
	\begin{equation}\label{eq:seriesunlbKe}
		u(t)=\sum_{\ell=0}^{+\infty}\sum_{m=-\ell}^{\ell}f_{\ell,m}\fe_\Phi(t;i\mu m-\ell(\ell+1))Y_{\ell,m},
	\end{equation}
where the equality holds in $H^s(\bS^2)$.
\end{thm}
\begin{rmk}
	Let us stress that if $\mu=0$, $s>\frac{7}{2}$ and $\Phi(\lambda)=\lambda^\alpha$ for some $\alpha \in (0,1)$, we recover \cite[Theorem 3.2]{d2014time}.
\end{rmk}
\begin{proof}
	Let us first prove that the series in \eqref{eq:seriesunlbKe} is convergent in $L^2(d\sigma;\C)$. Indeed, by Parseval's identity, it is only sufficient to show that
	\begin{equation*}
		\sum_{\ell=0}^{+\infty}\sum_{m=-\ell}^{\ell}|f_{\ell,m}\fe_\Phi(t;i\mu m-\ell(\ell+1))|^2<+\infty.
	\end{equation*}
	This is clearly true since $|\fe_\Phi(t;i\mu-\ell(\ell+1))|\le 1$ and $f \in L^2(d\sigma;\C)$. Actually, since
	\begin{equation*}
		A_\ell(u(t))=\sum_{m=-\ell}^{\ell}|f_{\ell,m}\fe_\Phi(t;i\mu m-\ell(\ell+1))|^2\le A_\ell(f),
	\end{equation*}
	we clearly have
	\begin{equation*}
		\sum_{\ell=0}^{+\infty}\ell^{2s}A_\ell(u(t))\le \sum_{\ell=0}^{+\infty}\ell^{2s}A_\ell(f)<+\infty,
	\end{equation*}
	thus $u(t) \in H^s(\bS^2)$ for any $t \ge 0$. This also proves that for any $t \ge 0$ it holds
	\begin{equation}\label{eq:ineqnorms}
		\Norm{u(t)}{H^s(\bS^2)}\le \Norm{f}{H^s(\bS^2)}.
	\end{equation}
	Moreover, observe that it holds
	\begin{align}\label{eq:seriesabsunif}
		\begin{split}
		\sum_{\ell=0}^{+\infty}\sum_{m=-\ell}^{\ell}|f_{\ell,m}\fe_\Phi(t;i\mu m-\ell(\ell+1))Y_{\ell,m}(x)|
		&\le \sum_{\ell=0}^{+\infty}\sqrt{\frac{(2\ell+1)A_\ell(f)}{2\pi}}<+\infty,
	\end{split}
	\end{align}
	where we used the Cauchy-Schwarz inequality and equality \eqref{eq:additionthmspecial}, and the convergence is guaranteed by Lemma \ref{lem:seriesconv}, since $s-\frac{1}{2}>\frac{1}{2}$.
%
%
%
%
By Weiestrass M-test, this guarantees that \eqref{eq:seriesabsunif} is absolutely and uniformly convergent for $t \ge 0$ and $x \in \bS^2$.

Now let us prove that $u \in C(\R_0^+;H^s(\bS^2))$. To do this, let us consider, for any $t \ge 0$ and $h \ge -t$, the power spectrum of $u(t+h)-u(t)$, given by
\begin{equation*}
	A_\ell(u(t+h)-u(t))=\sum_{m=\ell}^{\ell}|f_{\ell,m}|^2|\fe_\Phi(t+h;i\mu m-\ell(\ell+1))-\fe_\Phi(t;i\mu m-\ell(\ell+1))|^2\,.
\end{equation*}
Since $A_\ell(u(t+h)-u(t))) \le 4 A_\ell(f)$, then 
\begin{equation*}
	\sum_{\ell=0}^{+\infty}(1+\ell^{2s})A_\ell(u(t+h)-u(t)) \le 4\sum_{\ell=0}^{+\infty}(1+\ell^{2s})A_\ell(f)<+\infty.
	\end{equation*}
Thus, by Parseval's identity and the dominated convergence theorem, we have
\begin{align*}
	\lim_{h \to 0}\Norm{u(t+h)-u(t)}{H^s(\bS^2)}^2&=\sum_{\ell=0}^{+\infty}(1+\ell^{2s})\lim_{h \to 0}A_\ell(u(t+h)-u(t))=0
\end{align*}
which proves that $u \in C(\R^+_0; H^s(\bS^2))$.

Next, let us prove that $u(t) \in H^{s+2}(\bS^2)$ for any $t>0$. Indeed, by \eqref{eq:ineq1}, we have
\begin{align*}
	\sum_{\ell=1}^{+\infty}&\ell^{2(s+2)}\sum_{m=-\ell}^{\ell}|f_{\ell,m}\fe_\Phi(t;i\mu m-\ell(\ell+1))|^2
	\le C\sum_{\ell=1}^{+\infty}\ell^{2s}A_\ell(f)<+\infty.
\end{align*}
Analogously, for fixed $t>0$ and $h>-\frac{t}{2}$ we have 
\begin{equation*}
	\sum_{\ell=1}^{+\infty}\ell^{2(s+2)}A_\ell(u(t+h)-u(t))\le C \sum_{\ell=1}^{+\infty}\ell^{2s}A_\ell(f)<+\infty,
\end{equation*}
which implies, by a direct application of the dominated convergence theorem, that $u \in C(\R^+,H^{s+2}(\bS^2))$.

Now we  show that $\pdsup{u}{t}{\Phi}$ is well defined in $C(\R_0^+;H^s(\bS^2))$. Indeed, thanks to the uniform convergence of \eqref{eq:seriesunlbKe} for $t \ge 0$ and \cite[Theorem 7.16]{rudin1976principles}, we have
\begin{align*}
	\int_0^{t}&(u(\tau,x)-f(x))\overline{\nu}_\Phi(t-\tau)d\tau\\
	&=\sum_{\ell=0}^{+\infty}\sum_{m=-\ell}^{\ell}f_{\ell,m}\left(\int_0^{t}\overline{\nu}_\Phi(t-\tau)(\fe_\Phi(\tau;i\mu m-\ell(\ell+1))-1)d\tau\right)Y_{\ell,m}(x).
\end{align*}
Now observe that, by Proposition \ref{prop:exp},
\begin{align}\label{eq:phiderpre}
	\sum_{\ell=0}^{+\infty}&\sum_{m=-\ell}^{\ell}f_{\ell,m}\der{}{t}\left(\int_0^{t}\overline{\nu}_\Phi(t-\tau)(\fe_\Phi(\tau;i\mu m-\ell(\ell+1))-1)d\tau\right)Y_{\ell,m}(x)\\
	&=\sum_{\ell=0}^{+\infty}\sum_{m=-\ell}^{\ell}f_{\ell,m}(i\mu m-\ell(\ell+1))\fe_\Phi(t;i\mu m-\ell(\ell+1)) Y_{\ell,m}(x).\notag
\end{align}
Fix $t_0>0$ and notice that for any $t\ge t_0$ by Cauchy-Schwarz inequality and Lemma \ref{lem:seriesconv}
\begin{equation}\label{eq:unif1}
		\sum_{\ell=1}^{+\infty}\sum_{m=-\ell}^{\ell}|f_{\ell,m}im\fe_\Phi(t;i\mu m-\ell(\ell+1))Y_{\ell,m}(\theta,\varphi)|
		\le C\sum_{\ell=1}^{+\infty}\ell^{-\frac{1}{2}}\sqrt{A_\ell(f)}<+\infty
\end{equation}
and 
\begin{equation}\label{eq:unif2}
	\sum_{\ell=0}^{+\infty}\sum_{m=-\ell}^{\ell}|f_{\ell,m}\fe_\Phi(t;i\mu m-\ell(\ell+1))\ell(\ell+1)Y_{\ell,m}(\theta,\varphi)|
	\le C\sum_{\ell=0}^{+\infty}\ell^{\frac{1}{2}}\sqrt{A_\ell(f)}<+\infty.
\end{equation}
Combining \eqref{eq:unif1} and \eqref{eq:unif2} we get the absolute and uniform convergence of \eqref{eq:phiderpre} for $t \ge t_0$ and $x \in \bS^2$. Thus, we get
\begin{align*}
	\dersup{u(t,\theta,\varphi)}{t}{\Phi}
	&=\der{}{t}\sum_{\ell=0}^{+\infty}\sum_{m=-\ell}^{\ell}f_{\ell,m}\left(\int_0^{t}\overline{\nu}_\Phi(t-\tau)(\fe_\Phi(\tau;i\mu m-\ell(\ell+1))-1)d\tau\right)Y_{\ell,m}(\theta,\varphi)\\
	&=\sum_{\ell=0}^{+\infty}\sum_{m=-\ell}^{\ell}f_{\ell,m}(i\mu m-\ell(\ell+1))\fe_\Phi(\tau;i\mu m-\ell(\ell+1))Y_{\ell,m}(\theta,\varphi),
\end{align*}
for any $t>0$, $\theta \in (0,\pi)$ and $\varphi \in (0,2\pi)$. Let us stress that up to now $\dersup{u(t,\theta,\varphi)}{t}{\Phi}$ has been proven to exist only for fixed $t>0$,  $\theta \in (0,\pi)$ and $\varphi \in (0,2\pi)$.

To ensure that $u$ is a strong solution of \eqref{eq:nonlocbKq}, we need to show that $\dersup{u}{t}{\Phi}$ is well-defined as the strong derivative in $H^s(\bS^2)$ of a Bochner integral, as stated in Definition \ref{defn:strongder}. To do this, we first need to prove that
\begin{equation*}
	\int_0^t \overline{\nu}_\Phi(t-\tau)(u(\tau)-f)d\tau \in H^s(\bS^2)
\end{equation*}
is well-defined as a Bochner integral. By Bochner's theorem (see \cite[Theorem 1.1.4]{arendt2011vector}), it suffices to observe that, by inequality \eqref{eq:ineqnorms},
\begin{align*}
	\int_0^t &\overline{\nu}_\Phi(t-\tau)\Norm{u(\tau)-f}{H^s(\bS^2)}d\tau\le 2 \Norm{f}{H^s(\bS^2)}\int_0^t \overline{\nu}_\Phi(\tau)d\tau<+\infty.
\end{align*}
Next, one has to prove that, for any $t>0$,
\begin{align}\label{eq:strongderlimit}
	\lim_{h \to 0}\Norm{\frac{1}{h}\left(\int_0^{t+h}\overline{\nu}_\Phi(t+h-\tau)(u(\tau)-f)d\tau-\int_0^{t}\overline{\nu}_\Phi(t-\tau)(u(\tau)-f)d\tau\right)-\dersup{u(t)}{t}{\Phi}}{H^s(\bS^2)}=0.
\end{align}
To do this, let us first observe that we have
\begin{align*}
	\frac{1}{h}&\left(\int_0^{t+h}\overline{\nu}_\Phi(t+h-\tau)(u(\tau)-f)d\tau-\int_0^{t}\overline{\nu}_\Phi(t-\tau)(u(\tau)-f)d\tau\right)-\dersup{u(t)}{t}{\Phi}\\
	&=\sum_{\ell=0}^{+\infty}\sum_{m=-\ell}^{\ell}f_{\ell,m}\left(\frac{\fE_\Phi(t+h;i\mu m-\ell(\ell+1))-\fE_\Phi(t;i\mu m-\ell(\ell+1))}{h}\right.\\
	&\quad \left.-(i\mu m-\ell(\ell+1))\fe_\Phi(t;i\mu m-\ell(\ell+1))\vphantom{\frac{\fE_\Phi(t+h;i\mu m-\ell(\ell+1))-\fE_\Phi(t;i\mu m-\ell(\ell+1))}{h}}\right)Y_{\ell,m},
\end{align*}
where for any $\lambda \in \C$ we set
\begin{equation*}
	\fE_\Phi(t;\lambda)=\int_0^t \overline{\nu}_\Phi(t-\tau)(\fe_\Phi(\tau;\lambda)-1)d\tau.
\end{equation*}
In particular, it is clear that $\fE_\Phi(t;\lambda) \in C^1(\R_0^+)$ with
\begin{equation*}
	\der{\fE_\Phi(t;\lambda)}{t}=\dersup{\fe_\Phi(t;\lambda)}{t}{\Phi}=\lambda \fe_\Phi(t;\lambda),
\end{equation*}
where we used Proposition \ref{prop:exp}, and thus, by Lagrange's theorem, there exists $\theta(h) \in [(t+h)\wedge t, (t+h)\vee t]$ such that
\begin{align*}
	\frac{1}{h}&\left(\int_0^{t+h}\overline{\nu}_\Phi(t+h-\tau)(u(\tau)-f)d\tau-\int_0^{t}\overline{\nu}_\Phi(t-\tau)(u(\tau)-f)d\tau\right)-\dersup{u(t)}{t}{\Phi}\\
	&=\sum_{\ell=0}^{+\infty}\sum_{m=-\ell}^{\ell}f_{\ell,m}(i\mu m-\ell(\ell+1))(\fe_\Phi(\theta(h);i\mu m-\ell(\ell+1))-\fe_\Phi(t;i\mu m-\ell(\ell+1)))Y_{\ell,m}.
\end{align*}
By definition, we have
\begin{align*}
	&\left\Vert\frac{1}{h}\left(\int_0^{t+h}\overline{\nu}_\Phi(t+h-\tau)(u(\tau)-f)d\tau-\int_0^{t}\overline{\nu}_\Phi(t-\tau)(u(\tau)-f)d\tau\right)-\dersup{u(t)}{t}{\Phi}\right\Vert_{H^s(\bS^2)}^2\\
	&=\sum_{\ell=0}^{+\infty}(1+\ell^{2s})\sum_{m=-\ell}^{\ell}|f_{\ell,m}|^2|i\mu m-\ell(\ell+1)|^2|\fe_\Phi(\theta(h);i\mu m-\ell(\ell+1))-\fe_\Phi(t;i\mu m-\ell(\ell+1))|^2,
\end{align*}
and we can use the dominated convergence theorem to take the limit as $h \to 0$, since by \eqref{eq:ineq1} and the fact that $f \in H^s(\bS^2)$
\begin{equation}\label{eq:ineq}
	|f_{\ell,m}|^2|i\mu m-\ell(\ell+1)|^2|\fe_\Phi(\theta(h);i\mu m-\ell(\ell+1))-\fe_\Phi(t;i\mu m-\ell(\ell+1))|^2 
	\le C|f_{\ell,m}|^2
\end{equation}
with
\begin{equation*}
	\sum_{\ell=0}^{+\infty}(1+\ell^{2s})\sum_{m=-\ell}^{\ell}|f_{\ell,m}|^2=\sum_{\ell=0}^{+\infty}(1+\ell^{2s})A_\ell(f)<+\infty.
\end{equation*}
Thus, by using the dominated convergence theorem, the fact that $\theta(h) \to t$ as $h \to 0$ and the continuity of $\fe_\Phi(t;i\mu m-\ell(\ell+1))$, we have
\begin{equation*}
	\lim_{h \to 0}\left\Vert\frac{1}{h}\left(\int_0^{t+h}\overline{\nu}_\Phi(t+h-\tau)(u(\tau)-f)d\tau-\int_0^{t}\overline{\nu}_\Phi(t-\tau)(u(\tau)-f)d\tau\right)-\dersup{u(t)}{t}{\Phi}\right\Vert_{H^s(\bS^2)}^2=0,
\end{equation*}
which proves that $\dersup{u(t)}{t}{\Phi} \in H^s(\bS^2)$ and it is well-defined according to Definition \ref{defn:strongder}. Substituting $h$ in place of $\theta(h)$ in \eqref{eq:ineq}, tells us that $\dersup{u}{t}{\Phi} \in C(\R^+;H^s(\bS^2))$.

Next, let us prove that $u$ satisfies the equalities in \eqref{eq:nonlocbKq}.
To do this, we work in spherical coordinates. First of all, let us observe that, for fixed $t_0>0$ and $0<\theta_1<\theta_2<\pi$, the quantity
	\begin{equation}\label{eq:firstder1}
		\sum_{\ell=0}^{+\infty}\sum_{m=-\ell}^{\ell}f_{\ell,m}\fe_\Phi(t;i\mu m-\ell(\ell+1))\ell \cot(\theta)Y_{\ell,m}(\theta,\varphi)
	\end{equation}
	is absolutely and uniformly convergent  for $t \ge t_0$, $\theta \in [\theta_1,\theta_2]$ and $\varphi \in [0,2\pi)$. Indeed, recalling that $\cot(\theta)$ is bounded as $\theta \in [\theta_1,\theta_2]$, using \eqref{eq:ineq1}, the Cauchy-Schwarz inequality and the addition formula \eqref{eq:additionthmspecial}, we get
	\begin{align*}
		\sum_{\ell=1}^{+\infty}&\sum_{m=-\ell}^{\ell}|f_{\ell,m}\fe_\Phi(t;i\mu m-\ell(\ell+1))\ell \cot(\theta)Y_{\ell,m}(\theta,\varphi)|
		\le C\sum_{\ell=1}^{+\infty}\ell^{-\frac{1}{2}}\sqrt{A_\ell(f)}<+\infty,
	\end{align*}
	where the convergence of the latter series is guaranteed by Lemma \ref{lem:seriesconv} as $s-\frac{1}{2}>-\frac{1}{2}$. Thus absolute and uniform convergence follows from Weierstrass's M-test.
%
%
	Analogously, for fixed $t_0>0$ the quantity
	\begin{equation}\label{eq:firstder2}
		\sum_{\ell=0}^{+\infty}\sum_{m=-\ell}^{\ell}f_{\ell,m}\fe_\Phi(t;i\mu m-\ell(\ell+1))\sqrt{(\ell-m)(\ell+m+1)}e^{-i\varphi}Y_{\ell,m+1}(\theta,\varphi)
	\end{equation}
	is absolutely and uniformly convergent for $\theta \in [0,\pi]$, $\varphi \in [0,2\pi)$ and $t \ge t_0$, as
	\begin{align*}
		\sum_{\ell=1}^{+\infty}&\sum_{m=-\ell}^{\ell}|f_{\ell,m}\fe_\Phi(t;i\mu m-\ell(\ell+1))\sqrt{(\ell-m)(\ell+m+1)}e^{-i\varphi}Y_{\ell,m+1}(\theta,\varphi)|\\
		&\le C\sum_{\ell=1}^{+\infty}\ell^{-\frac{1}{2}}\sqrt{A_\ell(f)}.
	\end{align*}
	The local absolute and uniform convergence of \eqref{eq:firstder1} and the absolute and uniform convergence of \eqref{eq:firstder2}, combined with 
	\begin{equation}\label{eq:dertheta}
		\pd{Y_{\ell,m}}{\theta}(\theta,\varphi)=\ell \cot(\theta)Y_{\ell,m}(\theta,\varphi)+\sqrt{(\ell-m)(\ell+m+1)}e^{-i\varphi}Y_{\ell,m+1}(\theta,\varphi)
	\end{equation}
	 tells us that, for any fixed $t_0>0$, the series
	\begin{equation}\label{eq:firstdertheta}
		\sum_{\ell=0}^{+\infty}\sum_{m=-\ell}^{\ell}f_{\ell,m}\fe_\Phi(t;i\mu m-\ell(\ell+1))\pd{Y_{\ell,m}}{\theta}(\theta,\varphi)
	\end{equation}
	is locally absolutely and uniformly convergent for $t \ge t_0$, $\theta \in (0,\pi)$ and $\varphi \in (0,2\pi)$.
	
Notice also that for any fixed $t_0>0$ the series
	\begin{equation}\label{eq:firstderphi}
		\sum_{\ell=0}^{+\infty}\sum_{m=-\ell}^{\ell}f_{\ell,m}\fe_\Phi(t;i\mu m-\ell(\ell+1))\pd{Y_{\ell,m}}{\varphi}(\theta,\varphi)=\sum_{\ell=0}^{+\infty}\sum_{m=-\ell}^{\ell}f_{\ell,m}im\fe_\Phi(t;i\mu m-\ell(\ell+1))Y_{\ell,m}(\theta,\varphi)
	\end{equation}
	is absolutely and uniformly convergent for $\theta \in [0,\pi]$, $\varphi \in [0,2\pi)$ and $t \ge t_0$, by \eqref{eq:unif1}. 
	Next, we want to prove that for any fixed $t_0>0$ the series
	\begin{equation}\label{eq:secondderphi}
		\sum_{\ell=0}^{+\infty}\sum_{m=-\ell}^{\ell}f_{\ell,m}\fe_\Phi(t;i\mu m-\ell(\ell+1))\pdsup{Y_{\ell,m}}{\varphi}{2}(\theta,\varphi)=-\sum_{\ell=0}^{+\infty}\sum_{m=-\ell}^{\ell}f_{\ell,m}m^2\fe_\Phi(t;i\mu m-\ell(\ell+1))Y_{\ell,m}(\theta,\varphi)
	\end{equation}
	is absolutely and uniformly convergent for $\theta \in [0,\pi]$, $\varphi \in [0,2\pi)$ and $t \ge t_0$. To do this, we observe that
	\begin{align*}
		\sum_{\ell=0}^{+\infty}&\sum_{m=-\ell}^{\ell}|f_{\ell,m}m^2\fe_\Phi(t;i\mu m-\ell(\ell+1))Y_{\ell,m}(\theta,\varphi)|
		\le K\sum_{\ell=0}^{+\infty}\ell^{\frac{1}{2}}\sqrt{A_\ell(f)}<+\infty.
	\end{align*}
	Analogously, for fixed $t_0>0$, the series
\begin{align}\label{eq:LBY}
	\begin{split}
	\sum_{\ell=0}^{+\infty}&\sum_{m=-\ell}^{\ell}f_{\ell,m}\fe_\Phi(t;i\mu m-\ell(\ell+1))\Delta_gY_{\ell,m}(\theta,\varphi)\\
	&=-\sum_{\ell=0}^{+\infty}\sum_{m=-\ell}^{\ell}f_{\ell,m}\ell(\ell+1)\fe_\Phi(t;i\mu m-\ell(\ell+1))Y_{\ell,m}(\theta,\varphi)
\end{split}
\end{align}
is absolutely and uniformly convergent for $\theta \in [0,\pi]$, $\varphi \in [0,2\pi)$ and $t \ge t_0$, by \eqref{eq:unif2}.
Moreover, by  
	\begin{equation}\label{eq:secondderthetawLB}
		\pd{}{\theta}\left(\sin(\theta)\pd{}{\theta}\right)Y_{\ell,m}=\sin(\theta)\left(\Delta_g Y_{\ell,m}+\frac{1}{\sin^2(\theta)}\pdsup{}{\varphi}{2}\right),
	\end{equation}
and the absolute and uniform convergence of \eqref{eq:secondderphi} and \eqref{eq:LBY}, we obtain that, for any fixed $t_0>0$, the series
\begin{equation}\label{eq:seconddertheta}
	\sum_{\ell=0}^{+\infty}\sum_{m=-\ell}^{\ell}f_{\ell,m}\fe_\Phi(t;i\mu m-\ell(\ell+1))\pd{}{\theta}\left(\sin(\theta)\pd{Y_{\ell,m}}{\theta}(\theta,\varphi)\right)
\end{equation}
is locally absolutely and uniformly convergent for $t \ge t_0$, $\theta \in (0,\pi)$ and $\varphi \in (0,2\pi)$.

Thanks to the local absolute and uniform convergence of \eqref{eq:seriesunlbKe}, \eqref{eq:firstdertheta}, \eqref{eq:firstderphi}, \eqref{eq:secondderphi} and \eqref{eq:seconddertheta}, we know that we can exchange the operator $\cG_\mu$ with the summation sign to achieve
\begin{align}\label{eq:Gentou}
	\begin{split}
	\cG_\mu u(t,\theta,\varphi)&=\sum_{\ell=0}^{+\infty}\sum_{m=-\ell}^{\ell}f_{\ell,m}\fe_\Phi(t,i\mu m-\ell(\ell+1))\cG_\mu Y_{\ell,m}(\theta,\varphi)\\
	&=\sum_{\ell=0}^{+\infty}\sum_{m=-\ell}^{\ell}f_{\ell,m}\fe_\Phi(t,i\mu m-\ell(\ell+1))(i\mu-\ell(\ell+1)) Y_{\ell,m}(\theta,\varphi)\\
	&=\dersup{u(t,\theta,\varphi)}{t}{\Phi},
\end{split}
\end{align}
for all $t>0$, $\theta \in (0,\pi)$ and $\varphi \in (0,2\pi)$.

Finally, let us prove the uniqueness of the solution. Since the equation is linear, we only need to prove uniqueness when $f=0$. To do this, let us assume $u$ is a strong solution of \eqref{eq:nonlocbKq}. Then, since $u \in C(\R^+;H^{s+2}(\bS^2))$ and $\displaystyle \dersup{u}{t}{\Phi} \in C(\R^+;H^s(\bS^2))$, it holds
\begin{equation}\label{eq:seriesuniq}
	u(t)=\sum_{\ell=0}^{+\infty}\sum_{m=-\ell}^{\ell}u_{\ell,m}(t)Y_{\ell,m}, \ \mbox{ with } \
	\sum_{\ell=0}^{+\infty}(1+\ell^{2s+2})\sum_{m=-\ell}^{\ell}|u_{\ell,m}(t)|^2<+\infty.
\end{equation}
It is not difficult to check that
\begin{equation}\label{eq:seriesdfuniq}
	\dersup{u}{t}{\Phi}(t)=\sum_{\ell=0}^{+\infty}\sum_{m=-\ell}^{\ell}\dersup{u_{\ell,m}}{t}{\Phi}(t)Y_{\ell,m} \  \mbox { and } \ \cG_\mu u(t)=\sum_{\ell=0}^{+\infty}\sum_{m=-\ell}^{\ell}(i\mu m-\ell(\ell+1))u_{\ell,m}(t)Y_{\ell,m}.
\end{equation}
Since $u$ satisfies \eqref{eq:nonlocbKq}, then it is clear that the coefficients $u_{\ell,m}$ must be solutions of
\begin{equation*}
	\begin{cases}
		\displaystyle \dersup{u_{\ell,m}(t)}{t}{\Phi}=(i\mu m-\ell(\ell+1))u_{\ell,m}(t) & t>0\\
		u_{\ell,m}(0)=0
	\end{cases}
\end{equation*}
which implies that $u_{\ell,m}\equiv 0$, concluding the proof.
\end{proof}
Let us underline that, differently from the classical case, we need some better regularity of the initial data and we observe only a \textit{partial} parabolic smoothing. Indeed, we are able to prove the theorem only for an initial data in $H^s(\bS^2)$ where $s>1$ and the solution $u(t) \in H^{s+2}(\bS^2)$ for any $t>0$, in place of the classical parabolic smoothing effect in which $u(t) \in C^\infty(\bS^2)$. This is not due to a technical limitation of the proof, since we are able to provide an example in which better regularity cannot be achieved.
\begin{prop}\label{prop:counter}
	Let $\Phi(\lambda)=\lambda^{\alpha}$ for some $\alpha \in (0,1)$. For all $s>1$ there exists a function $f \in H^{s}(\bS^2)$ such that for all $t>0$ and $\varepsilon>0$ the strong solution $u$ of \eqref{eq:nonlocbKq} satisfies $u(t) \not \in H^{s+2+\varepsilon}(\bS^2)$.
\end{prop}
\begin{proof}
	Fix $s>1$ and consider the function $f=\sum_{\ell=2}^{+\infty}f_\ell Y_{\ell,0}$, where 
	\begin{equation*}
		f_\ell=\frac{1}{\ell^{s+\frac{1}{2}}\log(\ell)} \ \mbox{ with } \ \sum_{\ell=2}^{+\infty}\frac{(1+\ell^{2s})}{\ell^{2s+1}\log^2(\ell)}<\infty.		
	\end{equation*}
	By Theorem~\ref{thm:specdecKolm} we have
	\begin{equation*}
		u(t)=\sum_{\ell=2}^{+\infty}\frac{E_\alpha(-\ell(\ell+1)t^{\alpha})}{\ell^{s+\frac{1}{2}}\log(\ell)}Y_{\ell,0}.		
	\end{equation*}
	Now consider any $\varepsilon>0$ and observe that
	\begin{align*}
		\sum_{\ell=2}^{+\infty}\frac{\ell^{2(s+2+\varepsilon)}E^2_\alpha(-\ell(\ell+1)t^{\alpha})}{\ell^{2s+1}\log^2(\ell)} &\ge \sum_{\ell=2}^{+\infty}\frac{\ell^{2(s+2+\varepsilon)}}{\ell^{2s+1}\left(1+\Gamma(1-\alpha)\ell(\ell+1)t^{\alpha}\right)^2\log^2(\ell)}\\
		&\ge C(t)\sum_{\ell=2}^{+\infty}\frac{1}{\ell^{1-2\varepsilon}\log^2(\ell)}=\infty,
	\end{align*}
	where we used the lower bound in \cite[Theorem $4$]{simon2014comparing}.
\end{proof}
In the following section, our aim is to exploit a stochastic representation of the solutions of \eqref{eq:nonlocbKq}.


\subsection{The time-changed spherical Brownian motion with longitudinal drift}\label{sec:sphericalB}
Let $W^\mu(t)$ be a spherical Brownian motion with longitudinal drift, $\Phi$ be as in \eqref{eq:Bernstein} and let $S_\Phi(t)$ be a subordinator independent of it.
\begin{defn}
	We call $W^\mu_\Phi(t):=W^\mu(L_\Phi(t))$ the \textit{time-changed spherical Brownian motion with longitudinal drift}, with \textit{parent process} $W^\mu(t)$ and time-change $L_\Phi(t)$.
\end{defn}
\begin{rmk}
	If we consider the spherical Brownian motion $W(t)=(\theta(t),\varphi(t))$ in spherical coordinates, then we can represent, according to the identification \eqref{eq:sphericalcoord}, $W^\mu_\Phi$ as
	\begin{equation*}
		W^\mu_\Phi(t)=(\theta_\Phi(t),\varphi_\Phi(t)+\mu L_\Phi(t)),
	\end{equation*} 
	where $\theta_\Phi(t):=\theta(L_\Phi(t))$ and $\varphi_\Phi(t):=\varphi(L_\Phi(t))$. Actually, by using the theory developed in \cite{kobayashi2011stochastic}, we can define $\theta_\Phi(t)$ and $\varphi_\Phi(t)$ as the solution of the time-changed stochastic differential equations
	\begin{align*}
		d\theta_\Phi(t)&=dw^{(1)}(L_\Phi(t))+\frac{1}{2}\cot(\theta_\Phi(t))dL_\Phi(t)\\
		d\varphi_\Phi(t)&=\frac{1}{\sin(\theta_\Phi(t))}dw^{(2)}(L_\Phi(t)).
	\end{align*}
\end{rmk}
According to the discussion in Subsection \ref{subs:Sub}, $W_\Phi(t)$ is not a Markov process, but it is still a semi-Markov one. 
Consider the family of operators $T^\Phi_t: f \in L^1(\bS^2) \mapsto E[T_{L_\Phi(t)}f] \in L^1(\bS^2)$, where $\{T_t\}_{t\ge0}$ is the extension of the semigroup induced by $W^\mu$ on $L^1(\bS^2)$ provided by Lemma \ref{lem:extendLp}. Once we restrict $T^\Phi_t$ to $L^2(\bS^2)$, we are able to provide the following result.
\begin{thm}\label{thm:specdecgen}
	Let $f \in L^2(\bS^2)$ with Fourier-Laplace coefficients $(f_{\ell,m})_{\substack{\ell=0,1,\dots \\ m=-\ell,\dots,\ell}}$ and let
	\begin{equation}\label{eq:condexp2.1}
		u(t,x)=\E_x[f(W^\mu_\Phi(t))], \ \forall t \ge 0, \ \forall x \in \bS^2.
	\end{equation}
	Then it holds
	\begin{equation}\label{eq:condexp2}
		u(t)=\sum_{\ell=0}^{+\infty}\sum_{m=-\ell}^{\ell}f_{\ell,m}\fe_\Phi(t;im\mu -\ell(\ell+1))Y_{\ell,m}, \ \forall t \ge 0.
	\end{equation}
	Furthermore, $u \in C(\R_0^+;L^2(d\sigma;\C)) \cap C(\R^+;H^2(\bS^2))$.
\end{thm}
\begin{proof}
	Let us consider $f \in L^2(\bS^2)$. Since $C^\infty(\bS^2)$ is dense in $L^2(\bS^2)$, let us consider a sequence $\{f_n\} \subset C^\infty(\bS^2)$ such that $f_n \to f$ in $L^2(\bS^2)$. By \cite[Theorem 2.1]{chen2017time}, we know that the function
	\begin{equation*}
		u_n(t,x)=T^{\Phi}_tf_n(x)=\E[T_{L_\Phi(t)}f_n(x)]
	\end{equation*}
is the unique strong solution of \eqref{eq:nonlocbKq} with initial data $f_n$. 
By the Doob-Dynkin lemma and the independence of $L_\Phi(t)$ and $W^\mu(t)$, we know that
	\begin{equation*}
		T_{L_\Phi(t)}f_n(x)=\E_x[f_n(W^\mu(L_\Phi(t)))|L_\Phi(t)]=\E_x[f_n(W^\mu_\Phi(t))|L_\Phi(t)]
	\end{equation*}
	thus we can rewrite, thanks to the tower property of the conditional expectation,
	\begin{equation*}
		u_n(t,x)=\E[\E_x[f_n(W^\mu(L_\Phi(t)))|L_\Phi(t)]]=\E_x[f_n(W^\mu_\Phi(t))], \quad t \ge 0, \ x \in \bS^2.
	\end{equation*}

 Let  us denote by $(f^n_{\ell,m})_{\substack{\ell=0,1,\dots \\ m=-\ell,\dots,\ell}}$ the Fourier-Laplace coefficients of $f_n$. By Theorem \ref{thm:specdecKolm}, we can write
\begin{equation*}
	u_n(t)=\sum_{\ell=0}^{+\infty}\sum_{m=-\ell}^{\ell}f^n_{\ell,m}\fe_\Phi(t;i\mu m-\ell(\ell+1))Y_{\ell,m}.
\end{equation*}
Define the function $u(t)$ as in \eqref{eq:condexp2}. Clearly, for any $t\ge 0$ it holds $u(t) \in L^2(\bS^2)$. We want to show that for any $t \ge 0$ it holds $\lim_{n \to +\infty}\Norm{u_n(t)-u(t)}{L^2(\bS^2)}=0$. This is clear for $t=0$. Let $t > 0$ and observe that, by Parseval's identity, we have
\begin{align*}
	\Norm{u_n(t)-u(t)}{L^2(\bS^2)}&=\sum_{\ell=0}^{+\infty}\sum_{m=-\ell}^{\ell}|f_{\ell,m}-f^n_{\ell,m}|^2|\fe_\Phi(t;i\mu m-\ell(\ell+1))|^2\\
	&\le \Norm{f-f_n}{L^2(\bS^2)}.
\end{align*}
Taking the limit as $n \to +\infty$ we get the desired equality.

Next, recall that, by Lemma \ref{lem:extendLp}, $T_t:L^2(\bS^2) \to L^2(\bS^2)$ is a strongly continuous, positivity preserving, sub-Markov non-expansive semigroup. Hence it is not difficult to check that $T_t^\Phi:L^2(\bS^2) \to L^2(\bS^2)$ constitute a strongly continuous family of positivity preserving sub-Markov non-expansive linear operators. By definition, we have $T^\Phi_t f_n=u_n(t)$ and then, by continuity, we have
\begin{equation*}
	u(t)=\lim_{n \to +\infty}u_n(t)=\lim_{n \to +\infty}T^\Phi_t f_n=T^\Phi_t f,
\end{equation*}
which is equality \eqref{eq:condexp2.1}. Since $\{T^\Phi_t\}_{t \ge 0}$ is a strongly continuous family of operators, then $u \in C(\R_0^+;L^2(\bS^2))$. To prove that $u \in C(\R^+;H^2(\bS^2))$, we only need to verify that
\begin{equation*}
	\lim_{h \to 0}[u(t+h)-u(t)]^2_{H^2(\bS^2)}=0.
\end{equation*}
for any $t>0$. This is clear once we notice that, for $h>-\frac{t}{2}$,
\begin{equation*}
	\ell^2|\fe_\Phi(t+h;i\mu m-\ell(\ell+1))-\fe_\Phi(t;i\mu m-\ell(\ell+1))|^2\le C
\end{equation*}
for some $C>0$, implying that $\ell^2\cA_\ell(t,h)\le A_\ell(f)$, since in this case we can use dominated convergence theorem to achieve
\begin{equation*}
	\lim_{h \to 0}[u(t+h)-u(t)]^2_{H^2(\bS^2)}=\sum_{\ell=0}^{+\infty}\ell^2\lim_{h \to 0}\cA_\ell(t,h)=0.
\end{equation*}
\end{proof}
As a straightforward consequence, we have the following result.
\begin{cor}\label{cor:strong}
Let $f\in H^s(\bS^2)$, $s>1$. Then the function $u(t,x)=T^\Phi_t f(x)=\E_x[f(W^\mu_\Phi(t))]$ is the unique strong solution of \eqref{eq:nonlocbKq}.
\end{cor}
\begin{rmk}
Theorem \ref{thm:specdecgen} extends, in some sense, the spectral decomposition result and the parabolic smoothing effect of Theorem \ref{thm:specdecKolm} to the case of functions $f \in L^2(\bS^2)$. Indeed, in Subsection \ref{subs:weak}, we will use the spectral decomposition result provided in Theorem \ref{thm:specdecgen} to prove existence and uniqueness of very weak solution to \eqref{eq:nonlocbKq}.
Arguing in the same way as in Proposition \ref{prop:counter}, setting $\Phi(\lambda)=\lambda^\alpha$, $\alpha\in(0,1)$, the function $\displaystyle f=\sum_{\ell=2}^\infty f_\ell Y_{\ell,0}$ with $f_\ell=\frac{1}{\ell^{\frac{1}{2}}\log(\ell)}$ belongs to $L^2(\bS^2)$ and, for all $t,\varepsilon>0$, $T^\Phi_t f \not \in H^{2+\varepsilon}(\bS^2)$.
\end{rmk}

For this reason, despite $W^\mu_\Phi$ is not a Markov process, we still refer to \eqref{eq:nonlocbKq} as the backward Kolmogorov equation of the process. However, we underline that such equation could not fully characterize $W^\mu_\Phi$, differently from what happens with Feller processes (see \cite{polito2018studies} for an example on the fractional Poisson process).
To provide the fundamental solution of the backward Kolmogorov equation \eqref{eq:nonlocbKq}, we show that the law of $W^\mu_\Phi(t)$ with $W^\mu_\Phi(0)=x$ admits density with respect to $\sigma$, which can be expanded in terms of spherical harmonics.
\begin{thm}\label{thm:Propdens} 
	For any $x\in\bS^2$ we denote by $(\theta_x,\varphi_x)$ its spherical coordinates. There exists a function $p_\mu^\Phi:\R_0^+ \times \bS^2 \times \bS^2 \to \R$ with the following properties;
	\begin{itemize}
		\item[(i)] For any $x \in \bS^2$, $t>0$ and any $A \in \cB(\bS^2)$ it holds
		\begin{equation*}
			\bP_x(W_\Phi(t) \in A)=\int_{A}p_\mu^{\Phi}(t,\cdot|x)d\sigma.
		\end{equation*}
	\item[(ii)] For any $x,y \in \bS^2$ such that $\theta_x\ne \theta_y, \pi-\theta_y$ and any $t>0$ it holds 
		\begin{equation}\label{eq:serpmuP}
		p_\mu^\Phi(t,y|x)=\sum_{\ell=0}^{+\infty}\sum_{m=-\ell}^{\ell}\fe_\Phi(t;i\mu m-\ell(\ell+1))Y_{\ell,m}(x)\overline{Y}_{\ell,m}(y).
	\end{equation}
	\item[(iii)] For any $t>0$ and $x,y \in \bS^2$ such that $\theta_x\ne \theta_y, \pi-\theta_y$, the following equality holds:
	\begin{equation*}
		p_\mu^\Phi(t,y|x)=\E[p_\mu(L_\Phi(t),y|x)]=p_{-\mu}^\Phi(t,x|y).
	\end{equation*}
	\item[(iv)] For fixed $x \in \bS^2$, it holds $p_\mu^\Phi(t,\cdot|x), p_\mu^\Phi(t,x|\cdot) \in L^2(\bS^2)$;
\item[(v)] For any function $f \in C(\bS^2)$, it holds
\begin{equation*}
	\lim_{t \to 0} \int_{\bS^2}p_\mu^\Phi(t,y|x)f(y)d\sigma(y)=f(x).
\end{equation*}
	\end{itemize}
\end{thm}
\begin{proof}
	Fix $x \in \bS^2$ and a Borel set $A \subset \bS^2$. Applying Theorem \ref{thm:specdecgen} on the indicator function of $A$, we know that
	\begin{align*}
		\bP_x(W_\Phi(t)\in A)
				&=\sum_{\ell=0}^{+\infty}\int_{A}\sum_{m=-\ell}^{\ell} \overline{Y}_{\ell,m}(y)\fe_\Phi(t;i\mu m-\ell(\ell+1))Y_{\ell,m}(x)d\sigma(y).
	\end{align*}
	Now we want to exchange the integral sign with the summation sign. To do this, let us fix $x \in \bS^2$ and $\varepsilon>0$ sufficiently small. Let 
	\begin{equation*}
		S_{x,\varepsilon}=\{x \in \bS^2: \ |\theta_x-\theta_y| < \varepsilon \mbox{ or } |\theta_x+\theta_y-\pi| < \varepsilon\}
	\end{equation*}
	and define $A_\varepsilon=A \setminus S_{x,\varepsilon}$. Arguing as before we have
	\begin{equation}\label{eq:beforeFubini}
		\bP_x(W_\Phi(t)\in A_\varepsilon)=\sum_{\ell=0}^{+\infty}\int_{A_\varepsilon}\sum_{m=-\ell}^{\ell} \overline{Y}_{\ell,m}(t)\fe_\Phi(t;i\mu m-\ell(\ell+1))Y_{\ell,m}(x)d\sigma(y).
	\end{equation}
	We write explicitly the expectation in $\fe_\Phi$ to achieve
	\begin{align}\label{eq:propVI}
		\begin{split}
		&\sum_{m=-\ell}^{\ell} \overline{Y}_{\ell,m}(y)\fe_\Phi(t;i\mu m-\ell(\ell+1))Y_{\ell,m}(x)\\
		&=\E\left[\sum_{m=-\ell}^{\ell}e^{(i\mu m-\ell(\ell+1))L_{\Phi}(t)}Y_{\ell,m}(x)\overline{Y}_{\ell,m}(y)\right]\\
		&=\E\left[\frac{2\ell+1}{4\pi}e^{-\ell(\ell+1)L_\Phi(t)}P_\ell(x \cdot y_{L_{\Phi}(t)}^{-\mu})\right],
	\end{split}
	\end{align}
	where we used \eqref{eq:additiondrift}. Furthermore, it is clear that $M_\varepsilon=\max_{y \in \bS^2 \setminus S_{x,\varepsilon}}|x \cdot y|<1$, and then, since $y_{L_{\Phi}(t)}^{-\mu} \in \bS^2 \setminus S_{x,\varepsilon}$ for any $t \ge 0$, we can use \cite[Theorem 8.21.2]{szeg1939orthogonal}, to obtain that 
	\begin{equation*}
		|P_\ell(x \cdot y_{L_{\Phi}(t)}^{-\mu})| \le C \ell^{-\frac{1}{2}},
	\end{equation*}
	where the constant $C>0$ is independent of $x,y,t$ and $\ell$. This, in particular, implies that
	\begin{equation*}
	\left|\sum_{m=-\ell}^{\ell}\overline{Y}_{\ell,m}(y)\fe_\Phi(t,i\mu m-\ell(\ell+1))Y_{\ell,m}(x)\right|\le C\ell^{-\frac{3}{2}},
	\end{equation*}
	hence we can use Fubini's theorem in \eqref{eq:beforeFubini} to get 
	\begin{equation*}
		\bP_x(W_\Phi(t)\in A_\varepsilon)=\int_{A_\varepsilon}p_\mu^\Phi(t,y|x)d\sigma(y),
	\end{equation*}
	where 
	\begin{equation}\label{eq:seriesp}
		p_\mu^\Phi(t,y|x)=\sum_{\ell=0}^{+\infty}\sum_{m=-\ell}^{\ell} \overline{Y}_{\ell,m}(y)\fe_\Phi(t;i\mu m-\ell(\ell+1))Y_{\ell,m}(x).
	\end{equation}	
	The previous argument tells us that \eqref{eq:seriesp} is convergent for any $t>0$, $x \in \bS^2$ and $y \in \bS^2 \setminus S_{x,\varepsilon}$ for any $\varepsilon>0$, which, in particular, implies that \eqref{eq:seriesp} is convergent for any $t>0$ and $x,y \in \bS^2$ with $\theta_y\notin\{\theta_x, \pi-\theta_x\}$. Furthermore, for any fixed $x \in \bS^2$ and $t>0$, $p_\mu^\Phi(t,\cdot|x)$ is the density of the positive measure $B \in \cB(\bS^2 \setminus S_{x,\varepsilon}) \mapsto \bP_x(W_\Phi^\mu(t) \in B)$, hence, being $\varepsilon>0$ arbitrary, $p_\mu^\Phi(t,y|x) \ge 0$ for $t>0$ and $x,y \in \bS^2$ with $\theta_y\notin\{\theta_x, \pi-\theta_x\}$. As a consequence, by the monotone convergence theorem, it holds
	\begin{align*}
		\bP_x(W_\Phi(t) &\in A)=\bP_x(W_\Phi(t) \in A\setminus\{y\in\bS^2: \theta_y\notin\{\theta_x, \pi-\theta_x\}\})\\
		&=\lim_{\varepsilon \to 0}\bP_x(W_\Phi(t) \in A_\varepsilon)=\lim_{\varepsilon \to 0}\int_{A_\varepsilon}p_\mu^\Phi(t,y|x)d\sigma(y)=\int_{A}p_\mu^\Phi(t,y|x)d\sigma(y).
	\end{align*}
	This ends the proof of properties (i), (ii). Property (iii) is a straightforward consequence of \eqref{eq:propVI}. To show property (iv) it is sufficient to observe 
	\begin{align*}
		\Norm{p_\mu^\Phi(t,\cdot|x)}{L^2(\bS^2)}&=\sum_{\ell=0}^{+\infty}\sum_{m=-\ell}^{\ell}|\fe_\Phi(t;i\mu m-\ell(\ell+1))|^2|{Y}_{\ell,m}(x)|^2
		\le C\sum_{\ell=0}^{+\infty}\ell^{-3}<+\infty,
\end{align*}
where we used again the addition formula \eqref{eq:additionthmspecial}.

Finally, to prove item (v), just observe that since $f \in C(\bS^2)$, it is also bounded, and we can use the dominated convergence theorem to achieve
\begin{equation*}
	\lim_{t \to 0}\int_{\bS^2}p_\mu^\Phi(t,y|x)f(y)d\sigma(y)=\lim_{t \to 0}\E_x[f(W^\mu_\Phi(t))]=\E_x[f(W^\mu_\Phi(0))]=f(x).
\end{equation*}
\end{proof}
\begin{rmk}
	If $\mu=0$ and $\Phi(\lambda)=\lambda^\alpha$ for some $\alpha \in (0,1)$, we recover the probability density function obtained in \cite[Remark 3.5]{d2014time}.
\end{rmk}
\begin{rmk}
	Let us stress that we cannot guarantee for the convergence of the series defining $p_\mu^\Phi(t,x|y)$ if $y$ is such that $\theta_y\in\{\theta_x, \theta_x-\pi\}$. If, however, $\mu=0$, then one can prove, with exactly the same strategy, that the series \eqref{eq:serpmuP} converges for any $t>0$, $x \in \bS^2$ and $y \in \bS^2 \setminus \{\pm x\}$. If $y=-x$, the series can be rewritten as
	\begin{equation*}
		p^\Phi(t,-x|x):=p_0^\Phi(t,-x|x)=\sum_{\ell=0}^{+\infty}(-1)^\ell \frac{2\ell+1}{4\pi}\fe_\Phi(t;-\ell(\ell+1)),
	\end{equation*}
	thus convergence holds if, for any fixed $t>0$, $\ell \in \N \mapsto (2\ell+1)\fe_\Phi(t;-\ell(\ell+1)) \in \R$ is definitely decreasing.
	
	In the case $\Phi(\lambda)=\lambda^\alpha$ and $\mu=0$, we are actually able to prove that the series \eqref{eq:serpmuP} diverges whenever $x=y$. Indeed, since $E_\alpha(-\ell(\ell+1)t^\alpha) \sim C_\alpha(\ell(\ell+1))^{-1}$, it holds
	\begin{equation*}
		p_\mu^\Phi(t,x|x)=\sum_{\ell=0}^{+\infty}\frac{2\ell+1}{4\pi}E_\alpha(-\ell(\ell+1)t^{\alpha})\ge C\sum_{\ell=1}^{+\infty}\frac{2\ell+1}{\ell(\ell+1)} \ge C\sum_{\ell=1}^{+\infty}\frac{1}{\ell}=+\infty.
	\end{equation*}
\end{rmk}		
	By (i)-(ii) of Theorem \ref{thm:Propdens} we know that for any function $f \in L^2(\bS^2)$ and $t>0$,
	\begin{equation*}
		\E_x[f(W_\Phi(t))]=\int_{\bS^2}f(y)p_\mu^\Phi(t,y|x)d\sigma(y).
	\end{equation*}	
Combining this equality with (v) of Theorem \ref{thm:Propdens} we get that $p_\mu^\Phi$ is the fundamental solution of \eqref{eq:nonlocbKq}. 	
Similarly, we can introduce a forward Kolmogorov equation for $W_\Phi^\mu$. To do this, we first have to show that,
despite the process is not Markov, we can still use $p_\mu^\Phi$ as a transition density from the initial state, as underlined by the following proposition.
\begin{prop}\label{prop:density}
	For any $t>0$, any $\mathfrak{P} \in \mathcal{P}(\bS^2)$ and any $A \in \cB(\bS^2)$ it holds
	\begin{equation*}
		\bP_{\mathfrak{P}}(W^\mu_\Phi(t) \in A)=\int_{A} \int_{\bS^2} p_\mu^\Phi(t,y|x)d\mathfrak{P}(x)d\sigma(y)
		=:\int_{A}  p_{\mu,\mathfrak{P}}^\Phi(t,y) d\sigma(y)
	\end{equation*}
\end{prop}
\begin{proof}
	Let $t>0$, $\mathfrak{P} \in \cP(\bS^2)$ and $A \in \cB(\bS^2)$. Then
	\begin{align*}
		&\bP_{\mathfrak{P}}(W^\mu_\Phi(t) \in A)=\E_{\mathfrak{P}}[1_A(W^\mu_\Phi(t))]=\E_{\mathfrak{P}}[\E[1_A(W^\mu_\Phi(t))|W^\mu_\Phi(0)]]\\
		&=\int_{\bS^2} \E_x[1_A(W^\mu_\Phi(t))]d\mathfrak{P}(x)=\int_{\bS^2}\int_{A} p_\mu^\Phi(t,y|x)d\sigma(y)d\mathfrak{P}(x)=\int_{A}  p_{\mu,\mathfrak{P}}^\Phi(t,y) d\sigma(y),
	\end{align*}
	where we used Fubini's theorem since the integrand is non-negative.
\end{proof}
In case $\mathfrak{P}\in\cP(\bS^2)$ admits a density $f\in L^2(\bS^2)$, we have, by item (iii) of Theorem \ref{thm:Propdens}, 
\begin{equation*}
	p_{\mu,\mathfrak{P}}^\Phi(t,y)=\E_y[f(W_\Phi^{-\mu})]
\end{equation*}
hence it belongs to $C(\R_0^+;L^2(\bS^2))\cap C(\R^+;H^2(\bS^2))$ and for any $t\ge0$
\begin{equation}\label{eq:densspec}
		p_{\mu,\mathfrak{P}}^\Phi(t)=\sum_{\ell=0}^{+\infty}\sum_{m=-\ell}^{\ell}f_{\ell,m}\fe_\Phi(t;-i\mu m-\ell(\ell+1))Y_{\ell,m},
\end{equation}
where $f_{\ell,m}$ are the Fourier-Laplace coefficients of $f$.
Furthermore, by Corollary \ref{cor:strong}, if $f\in H^s(\bS^2)$, $s>1$, then $p_{\mu,\mathfrak{P}}^\Phi$ is the unique strong solution of 
\begin{equation}\label{eq:nonlocfor}
		\begin{cases}
			\displaystyle \pdsup{u(t)}{t}{\Phi}=\cG_{-\mu} u(t) & t>0 \\
			\displaystyle u(0)=f.
		\end{cases}
\end{equation}
We refer to this as the forward Kolmogorov equation for $W^\mu_\Phi$, since we observe that $\cG_{-\mu}$ is the adjoint of $\cG_{\mu}$.

Concerning the stationary and limit distributions, the following proposition holds true.
\begin{prop}\label{prop:stat}
Let $\mathfrak{P}_U=\frac{\sigma}{4\pi}$. For any $A \in \cB(\bS^2)$, $t>0$ and any $\mathfrak{P} \in \cP(\bS^2)$ it holds
\begin{equation}\label{eq:propstat}
	\lim_{s \to +\infty}\bP_{\mathfrak{P}}(W^\mu_\Phi(s) \in A)=\frac{\sigma(A)}{4\pi}=\bP_{\mathfrak{P}_U}(W^\mu_\Phi(t) \in A).
\end{equation}
\end{prop}
\begin{proof}
		Arguing as in Proposition \ref{prop:density}, for any $A \in \cB(\bS^2)$  we have 
		\begin{equation*}
		\bP_{\mathfrak{P}}(W^\mu_\Phi(s) \in A)=\int_{\bS^2}\sum_{\ell=0}^{+\infty}\sum_{m=-\ell}^{\ell}\fe_\Phi(t;i\mu m-\ell(\ell+1))f_{\ell,m}Y_{\ell,m}(x)d\mathfrak{P}(x),
		\end{equation*}
		where $(f_{\ell,m})_{(\ell,m) \in \widetilde{\mathbb{Z}}}$ are the Fourier-Laplace coefficients of $1_A$. It is not difficult to check that, by \eqref{eq:ineq1}, the Cauchy-Schwarz inequality and the addition formula \eqref{eq:additionthmspecial}, for any $\ell \ge 1$, $t \ge 1$ and $y \in \bS^2$
		\begin{equation*}
			\sum_{m=-\ell}^{\ell}|f_{\ell,m}\fe_\Phi(t;-i\mu m-\ell(\ell+1))Y_{\ell,m}(x)|\le C\ell^{-\frac{3}{2}}\sqrt{A_\ell(f)}.
		\end{equation*}
		Hence, by Lemma \ref{lem:seriesconv} and the fact that $\sigma(\bS^2)=4\pi$, we can use the dominated convergence theorem to achieve 
		\begin{align*}
			\lim_{t \to +\infty}\bP_{\mathfrak{P}}(W^\mu_\Phi(s) \in A)&=\int_{\bS^2}\sum_{\ell=0}^{+\infty}\sum_{m=-\ell}^{\ell}\lim_{t \to +\infty}\fe_\Phi(t;i\mu m-\ell(\ell+1))f_{\ell,m}Y_{\ell,m}(x)d\mathfrak{P}(x)\\
			&=\frac{1}{\sqrt{4\pi}}\int_{\bS^2}f_{0,0}d\mathfrak{P}(x)=\frac{\sigma(A)}{4\pi},
		\end{align*}
		where we used the fact that since $Y_{0,0}=\frac{1}{\sqrt{4\pi}}$ then $f_{0,0}=\frac{\sigma(A)}{\sqrt{4\pi}}$ and $\mathfrak{P}$ is a probability measure. To handle the second part of equality \eqref{eq:propstat}, let us observe that the probability measure $\mathfrak{P}_U=\frac{\sigma}{4\pi}$ admits a density $f=\frac{1}{4\pi} \in L^2(\bS^2)$. Once we observe that $f_{\ell,m}=0$ for any $\ell \ge 1$ and $m=-\ell,\dots, \ell$, while $f_{0,0}=\frac{1}{\sqrt{4\pi}}$, we can use \eqref{eq:densspec} and Proposition \ref{prop:density} to get
		\begin{equation*}
			\bP_{\mathfrak{P}_U}(W^\mu_\Phi(t) \in A)=\int_{A}f_{0,0}Y_{0,0}(y)d\sigma(y)=\frac{\sigma(A)}{4\pi},
		\end{equation*}
		that concludes the proof.
\end{proof}

The latter proposition proves the ergodicity in total variation of the process $W^\mu_\Phi$ for an arbitrary ${\rm Law}(W^\mu_{\Phi}(0))$. However, we do not have any information about the speed of convergence. We give now a partial result, with respect to a suitable distance of probability measures, on the convergence rate.
\begin{prop}\label{prop:speed}
		Assume that $\Phi(\lambda)=\lambda^\alpha \cL(\lambda)$, where $\alpha \in [0,1)$ and $\cL(\lambda)$ satisfies $\lim_{\lambda \to 0^+}\frac{\cL(C\lambda)}{\cL(\lambda)}=1$ for all $C \ge 1$. Then
		\begin{itemize}
			\item[(i)] For all $s>1$, $t \ge 1$, and $x \in \bS^2$,
			\begin{equation}\label{eq:uniformuplow}
				\sup_{\substack{f \in H^s(\R^d) \\ \Norm{f}{H^s} \le 1}}\left|\E_x[f(W_\Phi^\mu(t))]-\frac{1}{4\pi}\int_{\bS^2}f(x)d\sigma(x)\right| \le \frac{\overline{C}(\sqrt{\zeta(2s-1)}+\sqrt{\zeta(2s)})}{\sqrt{2\pi}}t^{-\alpha}\cL(1/t), 
			\end{equation}
			where $\zeta$ is the Riemann zeta function and $\overline{C}:=\sup_{t \ge 1}\frac{\fe_\Phi(t;-2)}{\cL(1/t)}t^{\alpha}<\infty$.
			\item[(ii)] For all $s\ge0$ and $\varepsilon \in \left(0,\frac{\pi}{2}\right)$ if $x=(\theta_x,\varphi_x)$ with $\theta_x \not \in \left(\frac{\pi}{2}-\varepsilon,\frac{\pi}{2}+\varepsilon\right)$, then
			\begin{equation}\label{eq:uniformuplow2}
				\sqrt{\frac{3}{4\pi}}\underline{C}\cL(1/t)t^{-\alpha}\cos\left(\frac{\pi}{2}-\varepsilon\right)\le \sup_{\substack{f \in H^s(\R^d) \\ \Norm{f}{H^s} \le 1}}\left|\E_x[f(W_\Phi^\mu(t))]-\frac{1}{4\pi}\int_{\bS^2}f(x)d\sigma(x)\right|, \ \forall t \ge 1,
			\end{equation}
			where $\underline{C}:=\inf_{t \ge 1}\frac{\fe_\Phi(t;-2)}{\cL(1/t)}t^{\alpha}>0$.
			\item[(iii)] If $\mu=0$, then for all $s\ge0$, $t \ge 1$, and $x \in \bS^2$,
			\begin{equation}\label{eq:uniformuplow3}
			\sqrt{\frac{3}{\pi}}\frac{\underline{C} \cL(1/t)}{4}t^{-\alpha}\le \sup_{\substack{f \in H^s(\R^d) \\ \Norm{f}{H^s} \le 1}}\left|\E_x[f(W_\Phi^\mu(t))]-\frac{1}{4\pi}\int_{\bS^2}f(x)d\sigma(x)\right|,
			\end{equation}
			where $\underline{C}:=\inf_{t \ge 1}\frac{\fe_\Phi(t;-2)}{\cL(1/t)}t^{\alpha}>0$.
		\end{itemize}
\end{prop}
\begin{proof}
	Let us first prove item (i). It is not difficult to check that, by Lemma \ref{lem:asbeh} \begin{equation}\label{eq:uplow1}
		\overline{C}:=\sup_{t \ge 1}\frac{\fe_\Phi(t;-2)}{ \cL(1/t)}t^{\alpha}<\infty.
	\end{equation}
	 Consider any $f \in H^s(\R^d)$ with $\Norm{f}{H^s(\R^d)} \le 1$. Let $(f_{\ell,m})_{(\ell,m) \in \widetilde{\mathbb{Z}}}$ be its Fourier-Laplace coefficients and observe that, since $|\fe_\Phi(t;i\mu m-\ell(\ell+1))| \le \fe_{\Phi}(t;-\ell(\ell+1)) \le \fe_{\Phi}(t;-2)$, we have
	\begin{align*}
		&\left|\E_x[f(W_\Phi^\mu(t))]-\frac{1}{4\pi}\int_{\bS^2}f(x)d\sigma(x)\right|=\left|\sum_{\ell=1}^{+\infty}\sum_{m=-\ell}^{\ell}\fe_\Phi(t;im\mu-\ell(\ell+1))f_{\ell,m}Y_{\ell,m}(x)\right|\\
		&\le\sum_{\ell=1}^{+\infty}\fe_\Phi(t;-2)\sum_{m=-\ell}^{\ell}|f_{\ell,m}||Y_{\ell,m}(x)| \le\overline{C}t^{-\alpha}\cL(1/t)\sum_{\ell=1}^{+\infty}\sum_{m=-\ell}^{\ell}|f_{\ell,m}||Y_{\ell,m}(x)|.
	\end{align*}
	By the Cauchy-Schwarz inequality and the addition formula \eqref{eq:additionthmspecial}, for any $\ell \ge 1$, $t \ge 1$ and $x \in \bS^2$
	\begin{equation*}
		\sum_{m=-\ell}^{\ell}|f_{\ell,m}Y_{\ell,m}(x)|\le \frac{1}{\sqrt{2\pi}} (\ell^{\frac{1}{2}}+1)\sqrt{A_\ell(f)}.
	\end{equation*}
	Proceeding as in the proof of Lemma \ref{lem:seriesconv} and recalling that $\Norm{f}{H^s(\bS^2)} \le 1$, we finally achieve
	\begin{align*}
		\left|\E_x[f(W_\Phi^\mu(t))]-\frac{1}{4\pi}\int_{\bS^2}f(x)d\sigma(x)\right|&\le\frac{\overline{C}(\sqrt{\zeta(2s-1)}+\sqrt{\zeta(2s)})}{\sqrt{2\pi}}t^{-\alpha}\cL(1/t),
	\end{align*}
	where $\zeta$ is the Riemann Zeta function.
	Now we prove Item (ii).	First, observe that
	\begin{equation}\label{eq:uplow}
		\underline{C}:=\inf_{t \ge 1}\frac{\fe_\Phi(t;-2)}{\cL(1/t)}t^{\alpha}>0.
	\end{equation}
	 Fix $\varepsilon \in \left(0,\frac{\pi}{2}\right)$ and consider $x=(\theta_x,\varphi_x)$ such that $\theta_x \not \in \left(\frac{\pi}{2}-\varepsilon,\frac{\pi}{2}+\varepsilon\right)$. It is sufficient to prove the lower bound in \eqref{eq:uniformuplow2} for a test function $f \in H^s(\R^d)$ with $\Norm{f}{H^s(\R^d)} \le 1$. We consider $f=Y_{1,0}$, so that $\frac{1}{4\pi}\int_{\bS^2}f(x)d\sigma(x)=0$, and then, by Theorem \ref{thm:specdecgen} and \eqref{eq:Ylm},
	\begin{equation*}
		\left|\E_x[f(W_\Phi^\mu(t))]\right|=\sqrt{\frac{3}{4\pi}}|\fe_\Phi(t;-2)||\cos(\theta_x)| \ge \sqrt{\frac{3}{4\pi}}\underline{C} \cL(1/t)t^{-\alpha}\cos\left(\frac{\pi}{2}-\varepsilon\right).
	\end{equation*}
	Finally, we prove item (iii). Observe that item (ii) still holds with $\varepsilon=\frac{\pi}{4}$. Now consider $x=(\theta_x,\varphi_x)$ with $\theta_x \in \left[\frac{\pi}{4},\frac{3\pi}{4}\right]$. As test function we consider $f=Y_{1,1} \in H^s(\R^d)$, so that $\frac{1}{4\pi}\int_{\bS^2}f(x)d\sigma(s)=0$ and, by Theorem \ref{thm:specdecgen} and \eqref{eq:Ylm}, since $\mu=0$,
	\begin{equation*}
		\left|\E_x[f(W_\Phi^\mu(t))]\right|=\sqrt{\frac{3}{8\pi}}|\fe_\Phi(t;-2)||e^{i\varphi_x}||\sin(\theta_x)| \ge \sqrt{\frac{3}{\pi}}\frac{\underline{C} \cL(1/t)}{4}t^{-\alpha}.
	\end{equation*}
\end{proof}
\begin{rmk}
Let us stress that the set $\{f \in H^s(\R^d):\, \Norm{f}{H^s(\R^d)} \le 1\}$ is separating for $\mathcal{P}(\bS^2)$. Indeed, consider $\mathfrak{P}_1,\mathfrak{P}_2 \in \mathcal{P}(\bS^2)$ and observe that if $\int_{\bS^2} f d\mathfrak{P}_1=\int_{\bS^2} f d\mathfrak{P}_2$ for all $f \in H^s(\R^d)$ with $\Norm{f}{H^s(\R^d)} \le 1$, then the same relation holds for any $f\in C^\infty(\bS^2)$ and then for any indicator function by dominated convergence, which implies $\mathfrak{P}_1=\mathfrak{P}_2$. Hence the left hand side of \eqref{eq:uniformuplow} is a distance that metricizes the weak convergence in $\mathcal{P}(\bS^2)$.
As a consequence, the previous proposition tells us that, in the case $\Phi(\lambda)=\lambda^\alpha \cL(\lambda)$ and $\mu=0$, we have a uniform non-exponential upper and lower bound on the speed of convergence to the stationary state, that implies the anomalous diffusive behaviour of $W_\Phi^0$. For $\mu \not =0$, we still have an anomalous diffusive behaviour but we are not able to provide a uniform lower bound if the initial position is exactly on the equator.
\end{rmk}

\subsection{The time-nonlocal Kolmogorov equation: very weak solutions}\label{subs:weak}
Recall that, in Theorem \ref{thm:specdecgen}, we have shown that the family of operators $\{T_t^{\Phi}\}_{t \ge 0}$ acts on $L^2(\bS^2)$, while, for $f \in H^s(\bS^2)$, $s>1$, $u(t)=T_t^\Phi f$ is the strong solution of \eqref{eq:nonlocbKq}. It remains to investigate the connection between $u(t)$ and \eqref{eq:nonlocbKq} when $f\in L^2(\bS^2)$. In order to do this, we use  the notion of very weak solution as introduced in \cite{Chen2018}.
\begin{thm}\label{thm:WspecdecKolm}
	Assume that $f \in L^2(\bS^2)$ and
	\begin{equation*}
		f=\sum_{\ell=0}^{+\infty}\sum_{m=-\ell}^{\ell}f_{\ell,m}Y_{\ell,m}.
	\end{equation*}
	Then the time-nonlocal partial differential equation \eqref{eq:nonlocbKq} admits a unique Laplace transformable very weak solution $u$, in the sense that:
	\begin{itemize}
		\item[(i)] $u \in C(\R_0^+;L^2(\bS^2))$;
		\item[(ii)]  for almost all $x\in\bS^2$ there exists $\lambda_0\ge0$ such that for all $\lambda>\lambda_0$ it holds
		$$
		\int_0^\infty e^{-\lambda t}|u(t,x)|dt <\infty;
		$$
		\item[(iii)] for all $h \in H^2(\bS^2)$ it holds
		\begin{equation}\label{eq:weaksol}
			\der{}{t}\int_{\bS^2} h(x)(\overline{\nu}_\Phi \ast (u(\cdot,x)-f(x)))(t)d\sigma(x)=\int_{\bS^2} u(t,x) \mathcal{G}_{-\mu}h(x) d\sigma(x).
		\end{equation}
	\end{itemize}
	In particular, such a solution belongs to $C(\R^+;H^2(\bS^2))$ and is given by
	\begin{equation}\label{eq:WseriesunlbKe}
		u(t)=\sum_{\ell=0}^{+\infty}\sum_{m=-\ell}^{\ell}f_{\ell,m}\fe_\Phi(t;i\mu m-\ell(\ell+1))Y_{\ell,m},
	\end{equation}
	where the equality holds in $H^2(\bS^2)$.
\end{thm}

\begin{proof}
	We proceed as in the proof of \cite[Theorem 2.4]{Chen2018}. First of all, observe that, by Theorem \ref{thm:specdecgen}, the function $u$ defined in \eqref{eq:WseriesunlbKe} can be rewritten as $u(t,x)=\E_x[f(W_\Phi^\mu(t))]=\E[T_{L_\Phi(t)}f(x)]$, where $T_t$ is the semigroup generated by $W^\mu(t)$. The fact that $u\in C(\R_0^+;L^2(\bS^2))\cap C(\R^+;H^2(\bS^2))$ follows by Theorem \ref{thm:specdecgen}. 
We notice that, by Minkowski generalized inequality,  
\begin{align*}
\sqrt{\int_{\bS^2} \left(\int_0^t \overline{\nu}_\Phi(t-s)|u(s,x)-f(x)|\,ds\right)^2 d\sigma(x)}
\le 2 \max_{0\le s\le t} \Norm{u(s)}{L^2(\bS^2)} \int_0^t \overline{\nu}_\Phi(s)\,ds<\infty,
\end{align*}
hence $(\overline{\nu}_\Phi \ast (u(\cdot,x)-f(x)))(t)$ belongs to $L^2(\bS^2)$ and is absolutely convergent for almost all $x\in\bS^2$.
Next, we have
	\begin{align}
		&\int_{\bS^2} h(x)(\overline{\nu}_\Phi \ast (u(\cdot,x)-f(x)))(t)d\sigma(x)=  \int_{\bS^2} \int_0^t h(x) \overline{\nu}_\Phi (t-s) (u(s,x)-f(x)) ds d\sigma(x)\notag\\
		&=\int_{\bS^2} \int_0^t \int_0^\infty h(x) \overline{\nu}_\Phi (t-s) (T_\tau f(x)-f(x)) f_L(\tau, s) d\tau\, ds \, d\sigma(x).\label{eq:Wint1}
	\end{align}
	Notice that
	\begin{align*}
		& \int_0^t \int_0^\infty \int_{\bS^2} |h(x)| \overline{\nu}_\Phi (t-s) |T_\tau f(x)-f(x)| f_L(\tau, s) \,d\sigma(x) \, d\tau\, ds \\
		&\le \Norm{h}{L^2(\bS^2)}  \int_0^t \int_0^\infty \overline{\nu}_\Phi (t-s) \Norm{T_\tau f-f}{L^2(\bS^2)} f_L(\tau, s) \, d\tau\, ds\\
		&\le 2\Norm{h}{L^2(\bS^2)}  \Norm{f}{L^2(\bS^2)} \int_0^t \overline{\nu}_\Phi (s) \, ds<\infty,
	\end{align*}
	where we used also Lemma \ref{lem:extendLp}. Hence we can use Fubini's Theorem in \eqref{eq:Wint1} to get 
	\begin{align}
		&\int_{\bS^2} h(x)(\overline{\nu}_\Phi \ast (u(\cdot,x)-f(x)))(t)d\sigma(x)\notag \\
		&=\int_0^t \int_0^\infty \overline{\nu}_\Phi (t-s) f_L(\tau, s) \int_{\bS^2} h(x)  (T_\tau f(x)-f(x))  \, d\sigma(x) \,d\tau\, ds \notag \\
		&=\int_0^t \int_0^\infty \overline{\nu}_\Phi (t-s) f_L(\tau, s) \int_{\bS^2} \int_0^\tau  h(x)  \cG_\mu T_{v}f(x) dv  \, d\sigma(x) \,d\tau\, ds, \label{eq:Wint2}
	\end{align}
	where we also use the fact that, by Theorem \ref{eq:sKolmogorov}, the function $T_t f$ is a strong solution of \eqref{eq:sBK}.
	Let us work separately with the inner integrals. Observe that
	\begin{align*}
		&\int_0^\tau \int_{\bS^2}   |h(x)|  |\cG_\mu T_{v}f(x)|  \, d\sigma(x)\, dv\le \Norm{h}{L^2(\bS^2)}\int_0^\tau \Norm{T_{v}f}{H^2(\bS^2)} dv\\
		&\le\sqrt{\tau}\Norm{h}{L^2(\bS^2)}\sqrt{\int_0^\tau \sum_{\ell=0}^{\infty} \ell^2 e^{-2\ell(\ell+1)v}A_\ell(f) dv}
		\le \sqrt{\frac{\tau}{2}}\, \Norm{h}{L^2(\bS^2)}\Norm{f}{L^2(\bS^2)} <\infty\,,
	\end{align*}
	hence we can use again Fubini's Theorem to get 
	\begin{align*}
		\int_{\bS^2} \int_0^\tau  h(x)  \cG_\mu T_{v}f(x) dv  \, d\sigma(x)
		= \int_0^\tau \int_{\bS^2}  \cG_{-\mu} h(x)   T_{v}f(x)   \, d\sigma(x)\, dv\,.
	\end{align*}
	Combining this with \eqref{eq:Wint2}, we have
	\begin{align}
		&\int_{\bS^2} h(x)(\overline{\nu}_\Phi \ast (u(\cdot,x)-f(x)))(t)d\sigma(x)\notag \\
		&=\int_0^t \int_0^\infty \overline{\nu}_\Phi (t-s) f_L(\tau, s) \int_0^\tau \int_{\bS^2}  \cG_{-\mu} h(x)   T_{v}f(x) dv  \, d\tau \,d\sigma(x)\, ds. \label{eq:Wint3}
	\end{align}
	Before proceeding let us show that item (ii) holds with $\lambda_0=0$. Indeed, for any $\lambda >0$
	\begin{align*}
		\int_{\bS^2}\int_0^\infty e^{-\lambda t} |u(t,x)| dt d\sigma(x)\le \int_0^\infty  e^{-\lambda t} \Norm{u(t)}{L^2(\bS^2)} dt\le \frac1\lambda  \Norm{f}{L^2(\bS^2)}
	\end{align*}
	Furthermore, notice that 
	\begin{align}
		\int_0^\infty e^{-\lambda t} \int_{\bS^2}  \cG_{-\mu} h(x) u(t,x) \,d\sigma(x)\, dt= \frac{\Phi(\lambda )}\lambda  \int_{\bS^2} \cG_{-\mu} h(x) \int_0^\infty e^{-s\Phi(\lambda )} T_s f(x) \,ds\,d\sigma(x)\,,\label{eq:Wint4}
	\end{align}
	where we used \eqref{eq:Laplace} and the latter integral is absolutely convergent, since $\Phi(\lambda )>0$. On the other hand, by \cite[Proposition 1.6.4]{arendt2011vector}, \eqref{eq:Laplacenu} and \eqref{eq:Laplace}, 
	\begin{align}
		&\int_0^\infty e^{-\lambda t}\int_0^t  \overline{\nu}_\Phi (t-s) \int_0^\infty f_L(\tau, s) \int_0^\tau \int_{\bS^2}  \cG_{-\mu} h(x)   T_{v}f(x) dv  \, d\tau \,d\sigma(x)\, ds\,dt\notag\\
		&=\frac{\Phi(\lambda )}\lambda  \int_0^\infty \int_0^\infty e^{-\lambda t} f_L(\tau,t) \int_0^\tau \int_{\bS^2}  \cG_{-\mu} h(x)   T_{v}f(x) dv  \, d\tau \,d\sigma(x)\,dt \notag\\
		&=\frac{\Phi(\lambda )^2}{\lambda ^2}  \int_0^\infty e^{-\tau \Phi(\lambda )}  \int_0^\tau \int_{\bS^2}  \cG_{-\mu} h(x)   T_{v}f(x) dv  \,d\sigma(x)\, d\tau \notag \\
		&=\frac{\Phi(\lambda )^2}{\lambda ^2}  \int_0^\infty \int_{\bS^2} \cG_{-\mu} h(x)   T_{v}f(x) \int_v^\infty e^{-\tau \Phi(\lambda )} d\tau \,d\sigma(x)\,  dv \notag \\
		&=\frac{\Phi(\lambda )}{\lambda ^2}  \int_0^\infty \int_{\bS^2} \cG_{-\mu} h(x)   T_{v}f(x) e^{-v \Phi(\lambda )} \,d\sigma(x)\,  dv. \label{eq:Wint5}
	\end{align}
	Comparing \eqref{eq:Wint4} and \eqref{eq:Wint5}, by \cite[Corollary 1.6.5]{arendt2011vector} and the injectivity of the Laplace transform we achieve in \eqref{eq:Wint3}
	\begin{align*}
		&\int_{\bS^2} h(x)(\overline{\nu}_\Phi \ast (u(\cdot,x)-f(x)))(t)d\sigma(x)=\int_0^t\int_{\bS^2}  \cG_{-\mu} h(x) u(s,x) \,d\sigma(x)\,ds\,.
	\end{align*}
	Differentiating both sides of the equality, we get item (iii). Finally, uniqueness follows as in \cite[Theorem 2.4]{Chen2018}.
\end{proof}

\begin{rmk}
We underline that uniqueness holds even if we do not require the solution to be regular. Indeed, in item (i) we only require $u(t)\in L^2(\bS^2)$, that is enough to guarantee the well-posedness of the right hand side of item (iv). This is different from the classical notion of weak solution (see \cite[Chapter 7]{evans2022}) in which it is required that $u(t)\in H^1(\bS^2)$. For this reason, we use the name \emph{very weak solution} (see \cite{kinnunen2002} and references therein) in place of just weak solution.
\end{rmk}

\appendix

\section{Proof of Corollary \ref{cor:transition_density_standard}} \label{appCOR}

	Fix any $A \in \cB(\bS^2)$ and consider the indicator function $1_{A} \in L^2(\bS^2)$. By Theorem \ref{eq:sKolmogorov}, the solution of \eqref{eq:sBK} in $C^1(\R^+; L^2(\bS^2))\cap C(\R_0^+;L^2(\bS^2))$ is given by
	\begin{equation*}
		u(t)=\sum_{\ell=0}^{+\infty}\sum_{m=-\ell}^{\ell}f_{\ell,m}e^{(i\mu m-\ell(\ell+1))t}Y_{\ell,m}, \quad \text{where} \quad 
		f_{\ell,m}=\int_{A}\overline{Y}_{\ell,m}d\sigma.
	\end{equation*}
	Now let us observe that the series
	\begin{equation*}
		\sum_{\ell=0}^{+\infty}\sum_{m=-\ell}^{\ell}e^{(i\mu m-\ell(\ell+1))t}Y_{\ell,m}(x)\overline{Y}_{\ell,m}(y)
	\end{equation*}
	is uniformly convergent for $x,y \in \bS^2$ and $t \ge t_0$ for any $t_0>0$, by using the upper bound \eqref{eq:uniformbound}. In particular, this implies that
	\begin{align}\label{eq:transdens1}
		\begin{split}
		u(t,x)&=\int_{A}\left(\sum_{\ell=0}^{+\infty}\sum_{m=-\ell}^{\ell}e^{(i\mu m-\ell(\ell+1))t}Y_{\ell,m}(x)\overline{Y}_{\ell,m}(y)\right)d\sigma(y)\\
		&:=\int_{A}p_\mu(t,y|x)d\sigma(x).
	\end{split}
	\end{align}
	On the other hand, since $\cG_\mu$ is the generator of $W^\mu(t)$, it holds
	\begin{equation}\label{eq:transdens2}
		u(t,x)=\E_x[1_A(W^\mu(t))]=\bP_x(W^\mu(t) \in A).
	\end{equation}
	Combining \eqref{eq:transdens1} and \eqref{eq:transdens2} we end the proof of identities \eqref{eq:density} and \eqref{eq:specdensity}.
	The equality \eqref{eq:adjointdensity} follows easily from \eqref{eq:specdensity} and \eqref{eq:additiondrift}.
\qed

\section{Proof of Lemma \ref{lem:extendLp}}\label{app:lemext}
Define the linear operator $T^*_t:C(\bS^2) \to \mathcal{M}(\bS^2)$ (where $\mathcal{M}(\bS^2)$ is the space of Radon measures on $\bS^2$) such that
	\begin{equation*}
		\int_{\bS^2} (T_tu) \, v \,  d\sigma=\int_{\bS^2} u \, d(T_t^*v), \ \forall u,v \in C(\bS^2).
	\end{equation*}
Let us identify $T^*_t$. To do this, recall that, by Corollary \ref{cor:transition_density_standard}, it holds
	\begin{align*}
		\int_{\bS^2} T_tu(x) \, v(x) \,  d\sigma(x)&=\int_{\bS^2} \left(\int_{\bS^2}p_\mu(t,y|x)u(y)d\sigma(y)\right) \, v(x) \,  d\sigma(x)\\
		&=\int_{\bS^2} \left(\int_{\bS^2}p_{-\mu}(t,x|y)v(x)d\sigma(x)\right) \, u(y) \,  d\sigma(y),
	\end{align*}
	where we used Fubini's theorem since
	\begin{equation*}
		\int_{\bS^2} \left(\int_{\bS^2}p_\mu(t,y|x)|u(y)|d\sigma(y)\right) \, |v(x)| \,  d\sigma(x) \le 4\pi \Norm{u}{L^\infty(\bS^2)}\Norm{v}{L^\infty(\bS^2)}.
	\end{equation*}
	This implies
	\begin{equation*}
		T^*_t v(x)=\int_{\bS^2}p_{-\mu}(t,x|y)v(y)d\sigma(y)=\E_x[v(W^{-\mu}(t))].
	\end{equation*}
	Let us show that $T^*_t$ can be extended to a sub-Markov operator $T^{\circledast}_t:L^1(\bS^2) \to L^1(\bS^2)$. Indeed, if we let $v \in C(\bS^2)$, it holds
	\begin{align*}
		\Norm{T^*_t v}{L^1(\bS^2)}&=\sup_{\substack{u \in C(\bS^2) \\ |u| \le 1}}\left|\int_{\bS^2}T^*_t v(x)u(x)d\sigma(x)\right|\\
		&=\sup_{\substack{u \in C(\bS^2) \\ |u| \le 1}}\int_{\bS^2} |v(x)||T_tu(x)|d\sigma(x)\le \Norm{v}{L^1(\bS^2)},
	\end{align*}
	hence we can extend $T^*_t$ to a bounded linear operator $T^{\circledast}_t:L^1(\bS^2) \to L^1(\bS^2)$. Furthermore, for any $v \in L^1(\bS^2)$ with $0 \le v \le 1$ there exists a sequence $v_n \in C^\infty(\bS^2)$ of functions such that $0 \le v_n \le 1$ and $v_n \to v$ in $L^1(\bS^2)$. In particular, by the continuity of $T_t^{\circledast}$, we have $T^{\circledast}_t v_n \to T^{\circledast}_t v$ in $L^1(\bS^2)$. Thus there exists a non relabelled subsequence $v_n$ such that
	\begin{align*}
		\lim_n T_t^{\circledast}v_n(x)= T_t^{\circledast}v(x), \qquad \lim_n v_n(x)=v(x)
	\end{align*}
	for almost any $x \in \bS^2$. Since $0 \le v_n \le 1$, we can use the dominated convergence theorem to achieve
	\begin{equation*}
		T^{\circledast}_tv(x)=\lim_{n \to +\infty}T^{\circledast}_tv_n(x)=\lim_{n \to +\infty}\E_x[v_n(W^{-\mu}(t))]=\E_x[v(W^{-\mu}(t))],
	\end{equation*}
	for almost any $x \in \bS^2$. This clearly implies that $0 \le T^{\circledast}_tv\le 1$. Finally, \cite[Lemma 1.45]{bottcher2013levy} tells us that $T_t$ admits the desired extension.
	
	Next, with the same arguments we adopted for $T^{\circledast}_t$, one can prove that for $f \in L^1(\bS^2) \cap L^\infty(\bS^2)$ it holds
	\begin{equation*}
		T_tf(x)=\E_x[f(W(t))], \ \forall t \ge 0, \ \mbox{ for almost any } x \in \bS^2.
	\end{equation*}
	Moreover, for $f \in L^1(\bS^2)$ with $f \ge 0$ let us consider the sequence $f_n=f \wedge n$, $n \in \N$. Then $f_n \uparrow f$ in $L^1(\bS^2)$ and thus $\lim_{n \to +\infty}T_tf_n=T_tf$.  Let us consider a subsequence, that we still denote by $f_n$, such that $\lim_{n \to +\infty}T_tf_n(x)=T_tf(x)$ for almost any $x \in \bS^2$. Observing that $f_n \in L^\infty(\bS^2)$ with $f_n \ge 0$, we have, by the monotone convergence theorem,
	\begin{equation*}
		T_t f(x)=\lim_{n \to +\infty}T_t f_n(x)=\lim_{n \to +\infty}\E_x[f_n(W(t))]=\E_x[f(W(t))]
	\end{equation*}
	for almost any $x \in \bS^2$. Next, if $f \in L^1(\bS^2)$ is real-valued, one can observe that
	\begin{equation*}
		\E_x[|f(W(t))|]=T_t|f|(x)<+\infty, \ \forall t \ge 0, \ \mbox{ for almost any }x \in \bS^2,
	\end{equation*}
	since $|f| \in L^1(\bS^2)$ and $T_t:L^1(\bS^2) \to L^1(\bS^2)$. Hence, once we write $f=f_+-f_-$, where $f_+=f \wedge 0$ and $f_-=-(f \vee 0)$, it holds
	\begin{equation*}
		\E_x[f(W(t))]=\E_x[f_+(W(t))]-\E_x[f_-(W(t))]=T_tf_+(x)-T_tf_-(x)=T_tf(x)
	\end{equation*}
	for almost any $x \in \bS^2$ and any fixed $t \ge 0$. Finally, if $f \in L^1(\bS^2)$ is complex-valued, the same argument holds separately for the real and imaginary parts. Taking the sum and using the linearity of $T_t$ we conclude the proof.
\qed
 
\bibliographystyle{plain}
\bibliography{biblio2}
\end{document}